\documentclass[onefignum,onetabnum]{siamart251216}

\usepackage{tikz-cd}
\newcommand{\Hone}{H_1}
\newcommand{\Hhyb}{H_{\text{hyb}}}
\newcommand{\Htwo}{H_2} 
\newcommand{\Hhybap}{\widetilde{\Hhyb}}
\newcommand{\Htwoap}{\widetilde{\Htwo}} 
\newcommand{\Honeap}{\widetilde{\Hone}} 
\renewcommand{\AA}{\mathcal{A}} 

\usepackage{amsmath}
\usepackage{amssymb}
\usepackage{lipsum}
\usepackage{amsfonts}
\usepackage{graphicx}
\usepackage{epstopdf}
\usepackage{algorithmic}
\ifpdf
  \DeclareGraphicsExtensions{.eps,.pdf,.png,.jpg}
\else
  \DeclareGraphicsExtensions{.eps}
\fi

\newsiamremark{remark}{Remark}
\newsiamremark{hypothesis}{Hypothesis}
\crefname{hypothesis}{Hypothesis}{Hypotheses}
\newsiamthm{claim}{Claim}
\newsiamremark{fact}{Fact}
\crefname{fact}{Fact}{Facts}

\usepackage{amsopn}

\usepackage{subcaption}
\usepackage{tikz}
\newcommand{\ket}[1]{ | #1 \rangle} 
 
\newcommand{\braket}[2]{ \langle #1| #2 \rangle}
\newcommand{\bigO}[1]{\mathcal{O}\left( #1 \right)}
\definecolor{darkgreen}{rgb}{0,0.5,0}

\headers{Quantum DD for Preconditioning the Finite Element Method}{E. Fressart, M. Nowak, and N. Spillane}

\title{Quantum Domain Decomposition for Preconditioning the Finite Element Method}

\author{Elise Fressart\footnotemark[2] \thanks{Thales CortAIx-Labs, Palaiseau, France 
  (\email{elise.fressart@thalesgroup.com} \email{michel.nowak@thalesgroup.com}).}
\and Michel Nowak\footnotemark[1]
\and Nicole Spillane\thanks{CNRS, CMAP, \'Ecole polytechnique, Institut Polytechnique de Paris, Palaiseau, France (\email{nicole.spillane@cnrs.fr}).}}

\ifpdf
\hypersetup{
  pdftitle={Quantum DD for preconditioning the Finite Element Method},
  pdfauthor={E. Fressart, M. Nowak, and N. Spillane}
}
\fi

\begin{document}

\maketitle

\begin{abstract}
Even in cases where quantum linear solvers provide significant speedup compared to their classical counterparts, their performance depends on some of the same parameters. In particular, the condition number of the matrix which is to be inverted is a decisive parameter. A well known classical, and now quantum, remedy is to precondition the linear system $A x = b$ by premultiplying it by a matrix $H$ in such a way that the condition number of $HA$ is significantly smaller than the condition number of $A$. 

In this work, we focus on a family of preconditioners called domain decomposition. First, we prove that it is feasible to apply quantum domain decomposition. We provide upper bounds for the block-encoding parameters of the Poisson problem discretized by the finite element method and preconditioned by the two-level Additive Schwarz preconditioner (one of the most fundamental domain decomposition techniques). From these bounds, we deduce the complexity of the quantum linear system solver.  Second, we focus on a particular choice of local solver within the domain decomposition preconditioner by applying recent work by [Deiml and Peterseim, \textit{Math. Comput.}, 2025] on the Bramble--Pasciak--Xu (BPX) preconditioner. Finally, we provide details on how the operators are implemented.
\end{abstract}


\tableofcontents

\section{Introduction}

The numerical solution of partial differential equations, or PDEs, plays a central role in a wide range of scientific and engineering applications, including fluid dynamics, electromagnetism, and solid mechanics. Standard approaches rely on spatial discretization, which typically leads to large-scale sparse linear systems whose size grows rapidly as higher accuracy is sought. For fine-grid discretizations, the resulting linear algebra problems can become prohibitively expensive to solve using classical (direct or iterative) methods, due to both computational complexity and memory constraints. This scalability bottleneck has motivated the exploration of alternative computational paradigms, including quantum algorithms for linear systems and PDEs. In \cite{Harrow_2009}, Harrow, Hassidim, and Lloyd introduced the first quantum algorithm for solving linear systems of equations (the HHL algorithm). The complexity of quantum linear system solvers (QLSS) has since been improved to a dependency that is linear in the matrix condition number and logarithmic in the inverse of the precision see \textit{e.g.}, \cite{zbMATH06070940,10.1145/3498331,  zbMATH06823707, PRXQuantum.3.040303, Gily_n_2019, Lin2020optimalpolynomial,PhysRevLett.122.060504} and \cite{morales2025quantumlinearsolverssurvey} for a survey of the different propositions. Variational approaches to quantum linear system solving have also been proposed for the Noisy Intermediate-Scale Quantum regime, known as NISQ, in which quantum resources are limited \cite{BravoPrieto2023variationalquantum}. 

Quantum linear solvers have already been applied to the linear systems arising from the finite element discretization of partial differential equations. It has been shown in \cite{PhysRevA.93.032324} that, for smooth solutions of the Poisson equation, a quantum advantage can be achieved for large problem dimensions. As in the classical setting, the performance of QLSSs depends on the condition number of the system matrix. In the finite element context, the condition number typically grows with the problem size, which may ultimately negate the expected quantum speedup.

In classical numerical linear algebra, this issue is typically fixed by preconditioning: a preconditioner $H \approx A^{-1}$ is applied to improve the conditioning of the original system $Ax = b$, leading to the preconditioned system $HA x = Hb$. If $H$ is well chosen, the preconditioned system is solved more efficiently than the original one. Quantum preconditioning strategies have been proposed by \cite{ bagherimehrab2025fastquantumalgorithmdifferential,PhysRevLett.110.250504, gu2025quantumsimulationhelmholtzequations,Low2026quantumlinearsystem,PhysRevA.98.062321,PhysRevA.104.032422}. In particular, \cite{PhysRevLett.110.250504} introduces a QLSS using a sparse approximate inverse (SPAI) preconditioner, while \cite{PhysRevA.98.062321} studies a circulant preconditioning approach.

Recently, \cite{Deiml_2025} proposed a quantum linear system solver based on the multilevel Bramble--Pasciak--Xu (BPX) preconditioner \cite{zbMATH04197323} \textit{via} a block-encoding of the preconditioned system in its symmetrized form. The BPX preconditioning strategy has also been incorporated into the Schrödingerization framework \cite{jin2025quantumpreconditioningmethodlinear}, where singular value amplification is used to extract the solution. In a related direction, \cite{childs2026quantumalgorithmsheterogeneouspdes} develops quantum algorithms for the neutron diffusion $k$-eigenvalue problem using BPX-based techniques and fast inversion methods. Quantum linear system solvers have also been integrated into multiscale numerical homogenization frameworks to accelerate the computation of local fine-scale problems \cite{balazi2026quantumenhancednumericalhomogenization}.

The present work is inspired by the, previously mentioned, article \cite{Deiml_2025}. Whereas \cite{Deiml_2025} considers a multigrid-type preconditioner \cite{zbMATH04197323,arXiv:2512.06166}, we focus on preconditioning by domain decomposition. We borrow several ideas from \cite{Deiml_2025} and also apply their quantum BPX algorithm within our own preconditioner. In classical numerical linear algebra, domain decomposition methods constitute a well-established framework for constructing scalable preconditioners for elliptic problems \cite{zbMATH02113718}. They rely on a decomposition of the computational domain into subdomains, where local problems are solved and combined through a coarse space correction to obtain global convergence. Classical theory shows that, with appropriately designed coarse spaces, such methods yield condition number bounds that are robust with respect to mesh size and coefficient heterogeneities \cite{zbMATH05029831,zbMATH05172548,zbMATH01592007,zbMATH05651315,spillane_geneo_2025,zbMATH06303461,zbMATH02113718}. The contribution of this work is to show that domain decomposition preconditioners can be incorporated into a QLSS.
This establishes the feasibility of implementing domain decomposition preconditioners in a quantum linear system framework, thereby complementing existing quantum preconditioning strategies based on multilevel, sparse approximate inverse and circulant techniques.
We note for completeness that \cite{busaleh2026mitigatingbarrenplateausdomain} utilizes domain decomposition concepts in variational quantum algorithms, albeit in a distinct setting and for a different objective to ours. Indeed, \cite{busaleh2026mitigatingbarrenplateausdomain} introduces a domain decomposition based strategy to mitigate Barren plateau–like phenomena in variational quantum algorithms. The approach relies on localizing the cost function by partitioning the spatial domain into overlapping subdomains, each associated with a localized parametrized quantum circuit and measurement operator. This improves trainability compared to global variational formulations.

The article is organized as follows. Section~\ref{sec:FEM} presents the finite element setting for the scalar elliptic problem. Section~\ref{sec:DD} introduces the concepts of domain decomposition and discusses the potential advantages of domain decomposition in the quantum framework (Section~\ref{sec:DD-adv}). A particular domain decomposition preconditioner, the two-level Additive Schwarz preconditioner is defined.  Spectral bounds for the resulting preconditioned operator are stated in Theorem~\ref{th:DDtheory}. Following ideas from \cite{Deiml_2025}, the preconditioned system is written in a symmetrized form. This way, the spectral bounds imply a bound on the condition number which is later needed for the analysis of the QLSS. In Section~\ref{sec:quantum}, we construct a block-encoding of the symmetric preconditioned system, describe its implementation and analyze the cost of inverting it by a quantum singular value transformation. Emphasis is put on the case where BPX is used to solve each local problem within the domain decomposition preconditioner. In Section~\ref{sec:numerical} we present a proof-of-concept numerical simulation of the quantum circuits. 

\paragraph{Notation}
For any integer $m$, we denote by $I_m$ the identity matrix in $\mathbb R^{m \times m}$. For any matrix $A$, the spectral condition number is defined by
\[
\kappa(A)=\|A\|_2\,\|A^+\|_2
=\frac{\sigma_{\max}(A)}{\sigma_{\min}^+(A)},
\]
where $A^+$ denotes the pseudoinverse and
$\sigma_{\min}^+(A)$ is the smallest nonzero singular value of $A$. In the special case where $A$ is a non-singular positive normal matrix (\textit{e.g.}, Hermitian positive definite), the singular values are also the eigenvalues so
\[
\kappa(A)=\frac{\lambda_{\max}(A)}{\lambda_{\min}(A)}.
\]
The conjugate transpose of $A$ is denoted by $A^\dagger$, its pseudoinverse is denoted by $A^+$. For real matrices, the transpose of $A$ is denoted by $A^\top$.

\section{Finite element setting}
\label{sec:FEM}
In the present article we focus on a particular linear system that arises from discretizing a simple PDE on a simple geometry. We are confident that the methods introduced could be extended to more difficult problems in the future. 
\subsection{Poisson equation}
In detail, we solve the Poisson equation in a domain $\Omega \in \mathbb{R}^d$ with homogeneous Dirichlet boundary condition. We make the assumption that $\Omega$ is a rectangle (if $d=2$), cuboid (if $d=3$) or, more generally, a hyperrectangle (any $d$). The source term $f$ is taken in the Lebesgue space $L^2(\Omega)$ of functions that are square integrable over $\Omega$. As is standard, we denote by $H_0^1(\Omega)$ the Sobolev space of functions in $L^2(\Omega)$ that vanish on the boundary of $\Omega$ and have square-integrable first derivatives. 

Then, the formulation of our problem is: Find $u \in H_0^1(\Omega)$ such that 
\begin{equation}
\label{eq:varf}
\int_\Omega \rho(x) \, \nabla u(x) \cdot \nabla v(x)  \mathrm d x = \int_\Omega f(x) v(x) \mathrm d x  \quad \forall v \in H_0^1(\Omega),
\end{equation}
where, $\rho$ is into the space of ${d \times d}$ symmetric positive definite matrices with measurable entries. We assume that there exist $\rho_{\min},\rho_{\max}>0$ such that
\begin{equation}
\label{eq:assumption-rho}
\rho_{\min} \eta^\top \eta \leq \eta^\top \rho(x)  \eta \leq \rho_{\max} \eta^\top \eta ~ \text{for all } \eta \in \mathbb{R}^d \text{ and almost all } x \in \Omega. 
\end{equation}
As is well-known, this is the weak formulation of the partial differential equation (PDE)
\begin{equation}
    \label{eq:pde}
\left\{
    \begin{aligned}
- \nabla \cdot \rho \nabla u &= f \quad \text{in } \Omega, \\
u &= 0 \quad \text{on } \partial \Omega,
\end{aligned}
\right.
\end{equation}
in which $\rho$ is the, possibly anisotropic, diffusion parameter. If $\rho = I_{d}$, the first line in the equation rewrites as $- \Delta u = f$.
\subsection{Finite element setting}
\label{sub:FEM}
In order to discretize the PDE, we introduce a Cartesian mesh $\mathcal{T}_L$ of $\Omega$. Throughout the article, the subscript $\cdot_L$ refers to quantities defined on the mesh $\mathcal{T}_L$. We apply the $\mathbb Q_1$ finite element method, meaning that we consider the restriction of the variational formulation \eqref{eq:varf} to a particular finite dimensional subspace $V_L$ of $H_0^1(\Omega)$. The $\mathbb Q_1$ finite element space is the space of functions that are globally continuous over $\Omega$, $d$-linear on each mesh element and that vanish on $\partial \Omega$. Formally, we set 
\[
V_L := \left\{u_L \in C^0(\bar{\Omega}) \cap H_0^1(\Omega);\,  {u_L}_{|K} \in \mathbb Q_1(K), \, \forall K \in \mathcal{T}_L \right\},
\] 
where 
\[
\mathbb Q_1(K) = \operatorname{span}\left\{
x_1^{\alpha_1} \cdots x_d^{\alpha_d}
:\alpha_i \in \{0,1\}
\right\} \text{ with } x= (x_1,\dots,x_d) \text{ for } x \in K.
\]

We let $\{ \phi_j \}_{1 \leq j \leq 
\mathcal{N}}$ be the standard nodal basis of  $V_L$, \textit{i.e.}, the set of functions that are $1$ at an interior vertex of the mesh and $0$ at all others. We then obtain the linear system of equations 
\begin{equation}
\label{eq:Axequalb}
     A x = b,
\end{equation}
where $A \in \mathbb R^{\mathcal N \times \mathcal N}$ and $b \in \mathbb R^{\mathcal N}$ are defined by
\begin{equation*}
    A_{ij} := \int_\Omega \rho \, \nabla \phi_i \cdot \nabla \phi_j \text{ for all }1 \leq i,j \leq \mathcal{N} \text{ and } b_i := \int_\Omega  f \phi_i \text{ for all } 1 \leq i \leq \mathcal{N}.
\end{equation*}
Note that the matrix $A$ is symmetric positive definite and that its order, $\mathcal N$, is the number of vertices of $\mathcal{T}_L$ that are in the interior of $\Omega$. 

\subsection{Factorized form of the linear system}

We follow the idea from \cite{Deiml_2025} to factorize $A$ into discrete gradient operators weighted by the (matrix-valued) coefficient $\rho$. We assume that $\rho$ is constant over each mesh element $K \in \mathcal{T}_L$.

We introduce the discontinuous piecewise $\mathbb Q_1$ space
\begin{equation}
\label{eq:QL}
Q_L := \left\{u_L \in L^2(\Omega);\, {u_L}_{|K} \in \mathbb Q_1(K), \ \forall K \in \mathcal{T}_L \right\},
\end{equation}
where each element $K$ carries $2^d$ local degrees of freedom. We equip $Q_L$ with an $L^2$-orthonormal basis consisting of $2^d$ functions
$\psi_{K,1}, \dots, \psi_{K,2^d}$ on each element $K \in \mathcal{T}_L$.

We next define $Q_L^d := Q_L \otimes \mathbb{R}^d$, and equip it with the basis functions $\psi_{K,k} \otimes e_s $ where $e_s$ ($s \in \{1, \dots, d\}$) is the $s$-th canonical basis vector in $\mathbb R^{d}$. 
Let also $C_L$ denote the matrix representation of the (elementwise) gradient operator
\[
\nabla : V_L \to Q_L^d.
\]

Then, as in \cite{Deiml_2025}, the stiffness matrix admits the decomposition
\begin{equation}
A = C_L^\top (D_\rho \otimes I_{2^d}) C_L,
\label{eq:decomposition_stiffness_matrix}
\end{equation}
where $D_\rho$ is a block-diagonal matrix with one block per element $K \in \mathcal{T}_L$, each block being the diffusion tensor $\rho_{|K} \in \mathbb{R}^{d \times d}$. The tensor product structure reflects the assumption that $\rho$ is constant over each element.

The factorized form of $A$ plays a crucial role in block-encoding the linear system. Before moving to the quantum framework we introduce domain decomposition preconditioners.  

\section{Domain decomposition}
\label{sec:DD}

 This section focuses on the construction of a particular domain decomposition method: the two-level Additive Schwarz preconditioner with inexact local solves. The presentation is for a simple geometry but most definitions carry over to more complex geometries. The spectrum of the Poisson problem preconditioned by additive Schwarz is bounded by Theorem~\ref{th:DDtheory} (which is well-known in the domain decomposition literature \cite{zbMATH02113718}). A split, or factorized, form of the preconditioner is proposed that is well-suited to its quantum implementation. In Section~\ref{sec:DD-adv}, the advantages of domain decomposition are discussed as well as its potential for accelerating quantum linear solvers. 

\subsection{Geometry of the domain and subdomains}
\label{sec:geom}
We introduce a partition of $\Omega$ into $N$ overlapping subdomains denoted $\Omega^{(i)}$. The subdomains $\Omega^{(i)}$ are open, mesh conforming, subsets of $\Omega$ which satisfy
\[
\bigcup\limits_{i=1}^N \Omega^{(i)} = \Omega. 
\]
Remark that, in order to satisfy the definition, it is clear that the subdomains must overlap, \textit{i.e.}, $\Omega^{(i)} \cap \Omega^{(j)} \neq 0$ for some indices $j$. It is natural to define local meshes $\mathcal{T}_L^{(i)}$ which consist of the restriction of $\mathcal{T}_L$ to the elements in $\Omega^{(i)}$ and to the vertices in $\overline{\Omega^{(i)}}$. For convenience in introducing the quantum framework, we next introduce a simple geometry (but we stress that this is not necessary for domain decomposition). We first introduce the subdomains $\Omega^{(i)}$ and then set $\Omega$ to be their union. 

We assume that each subdomain $\Omega^{(i)}$ is a unit hypercube discretized by a Cartesian mesh  $\mathcal{T}_L^{(i)}$ of size $2^{-L}$. Consequently, each subdomain contains $2^{dL}$ elements ($2^L$ in each coordinate direction). We further assume that there are $N_s$ subdomains along each direction $s= 1, \dots, d$, and that they overlap by a fixed number of mesh elements given by $2 \delta / 2^{-L}$. (This ensures that $\delta$ matches the standard overlap parameter used in the domain decomposition literature.) Figure~\ref{fig:overlapping_subdomains} shows an example of this geometry for $d=2$, $N_1 = 4$, $N_2 = 2$, $L = 3$, and $\delta = 2^{-L}$.

\begin{remark}
\label{rem:i-multiindex}
The subdomains can also be labeled by a multi-index $(i_1,\dots,i_d)$ to make the Cartesian tensor structure explicit, which will be useful in later sections. For example, subdomain $\Omega^{(7)}$ in Figure~\ref{fig:overlapping_subdomains}, is multi-indexed $(i_1, i_2) = (3,2)$. 
\end{remark}

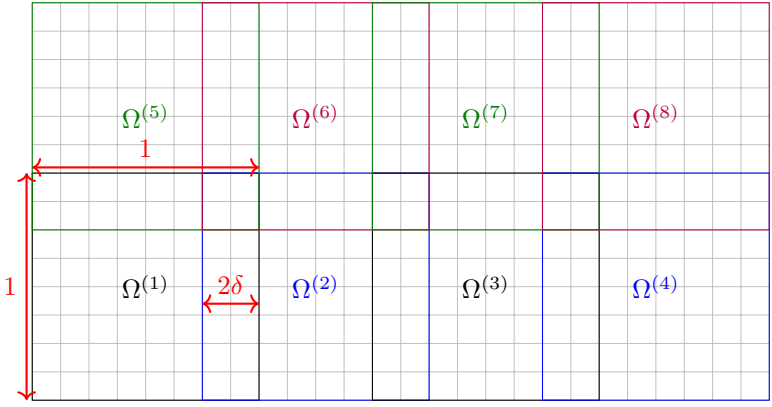
\begin{figure}
\begin{center}
  \begin{tikzpicture}[scale=0.75]
    \path [draw, help lines, opacity=.5]  (0, 0) grid[step=0.5] (13,7);
    \draw (0, 0) rectangle (4, 4) node [pos=0.5] {$\Omega^{(1)}$};
    \draw[blue] (3, 0) rectangle (7, 4) node [pos=0.5] {$\Omega^{(2)}$};
    \draw (6, 0) rectangle (10, 4) node [pos=0.5] {$\Omega^{(3)}$};
    \draw[blue] (9, 0) rectangle (13, 4) node [pos=0.5] {$\Omega^{(4)}$};
    
    \draw[darkgreen] (0, 3) rectangle (4, 7) node [pos=0.5] {$\Omega^{(5)}$};
    \draw[purple] (3, 3) rectangle (7, 7) node [pos=0.5] {$\Omega^{(6)}$};
    \draw[darkgreen] (6, 3) rectangle (10, 7) node [pos=0.5] {$\Omega^{(7)}$};
    \draw[purple] (9, 3) rectangle (13, 7) node [pos=0.5] {$\Omega^{(8)}$};
    
\draw[<->, red, thick] (3, 1.7) -- node[above] {$2\delta$} (4, 1.7);
 \draw[<->, red, thick] (0, 4.1) -- node[above] {$1$} (4, 4.1); 
    \draw[<->, red, thick] (-0.1, 0) -- node[left] {$1$} (-0.1, 4); 
  \end{tikzpicture}
 \end{center}
  \caption{Example of 8 overlapping subdomains ($d=2$, $N_1 = 4$, $N_2 = 2$, $L = 3$, and $\delta = 2^{-L}$).}
  \label{fig:overlapping_subdomains}
\end{figure}

The global domain is then 
\[
\Omega = \prod_{s=1}^d \left[ 0, N_s - 2\delta (N_s - 1) \right]. 
\]
The global mesh $\mathcal T_L$ is the Cartesian mesh of $\Omega$ with discretization step $2^{-L}$. The finite element space $V_L$ (already defined in Section~\ref{sub:FEM}) has $\mathcal N$ degrees of freedom (\textit{i.e.}, unknowns) at the interior mesh vertices with $\mathcal N = \prod_{s=1}^d \left[2^L (N_s - 2 \delta (N_s-1)) -1\right]$ . This is also the order of the matrix $A $ in the linear system \cref{eq:Axequalb}. 

We next introduce in each subdomain a local finite element space 
\[
V_L^{(i)} := \left\{u_L \in  C^0(\overline{\Omega^{(i)}}); \, u_L = 0 \text{ on } \partial \Omega \text{ and }  {u_L}_{|K} \in \mathbb Q_1(K), \,   \forall K \in \mathcal{T}_L^{(i)} \right\}.
\] 
The dimension of each $V_L^{(i)}$ is $(2^L-1)^{d}$ (the number of inner vertices of $\Omega^{(i)}$).
There is a natural embedding of $V_L^{(i)}$ into $V_L$ since the functions in $V_L^{(i)}$ can be seen as restrictions of certain functions in $V_L$. To make things concrete, we also introduce this interpolation operator in the algebraic formalism. The degrees of freedom (interior vertices) of $\Omega$ are numbered from $1$ to $\mathcal N$. The degrees of freedom (interior vertices) of $\Omega^{(i)}$ are a subset of vertices that we can denote by $\omega_i$. The restriction matrix $R_L^{(i)} \in \mathbb{R}^{(2^L - 1)^d \times \mathcal N}$ is
\[
R_L^{(i)} := I_\mathcal N(\omega_i, :),
\]  
by which it is meant that we select all columns and only rows that are in $\omega_i$ . 
Matrix $R_L^{(i)}$ is the representation of the restriction operator $V_L \rightarrow V_L^{(i)}$. Its transpose ${R_L^{(i)}}^\top$ represents the prolongation by zero operator $V_L^{(i)} \rightarrow V_L$, which we also call the natural embedding $V_L^{(i)} \hookrightarrow V_L$.

\textit{Example: } In the one-dimensional case with 2 subdomains, $L = 2$, and overlap $\delta=1/4$, the extension matrices are
\begin{equation*}
{R_L^{(1)}}^\top = 
\begin{pmatrix}
1 & 0 & 0 \\ 
0 & 1 & 0 \\ 
0 & 0 & 1 \\ 
0 & 0 & 0 \\ 
0 & 0 & 0 \\ 
\end{pmatrix} \text{ and }
{R_L^{(2)}}^\top = 
\begin{pmatrix}
0 & 0 & 0  \\ 
0 & 0 & 0  \\ 
1 & 0 & 0  \\ 
0 & 1 & 0  \\ 
0 & 0 & 1  \\ 
\end{pmatrix}.
\end{equation*}

The second last ingredient for defining domain decomposition preconditioners is to define, for each $i$, a local solver that acts in the local subspaces $V_L^{(i)}$. In a first step, we choose the exact local solver ${A^{(i)}}^{-1}$ with  
\[
A^{(i)} := R_L^{(i)} A {R_L^{(i)}}^\top = A(\omega_i, \omega_i), 
\]
by which it is meant that we select the rows and columns of $A$ that are in $\omega_i$. Later on, these $A^{(i)}$ will be replaced by approximations.

It is to be noted that $A^{(i)} $ is also the matrix associated with the finite element discretization, in $V_L^{(i)}$, of the PDE 
\[
\left\{
\begin{aligned}
- \nabla \cdot \rho \nabla u &= f \quad \text{in } \Omega^{(i)}, \\
u &= 0 \quad \text{on } \partial \Omega^{(i)}.
\end{aligned}
\right.
\]

\subsection{Domain Decomposition preconditioners}

Having introduced the interpolation operators (${R_L^{(i)}}$, ${R_L^{(i)}}^\top$) between local and global functions, as well as the local solvers ${A^{(i)}}^{-1}$ we can next define the one level Additive Schwarz preconditioner as
\begin{equation}
\Hone := \sum_{i=1}^N {R_L^{(i)}}^\top {A^{(i)}}^{-1} R_L^{(i)}. 
\end{equation}
To say it with words, the inverse of $A$ is approximated by a sum of local inverses $ {A^{(i)}}^{-1}$ in the subdomains. This approximation is, in general, not sufficient to guarantee that the eigenvalues of the preconditioned operator $\Hone A$ are bounded independently from the number of subdomains (what is called the scalability property). For this reason, an additional correction is performed called the coarse correction. The coarse correction solves one more problem in a space called the coarse space (which is a subset of $V_L$). We choose to parametrize the coarse space by a matrix $Z$, the columns of which form a basis of the coarse space. The coarse problem is then $Z^\top A Z$, and the two-level Additive Schwarz preconditioner is defined by 
\begin{equation}
\label{eq:HASM2}
\Htwo := \sum_{i=1}^N {R_L^{(i)}}^\top {A^{(i)}}^{-1} R_L^{(i)} + Z (Z^\top A Z)^{-1} Z^\top.
\end{equation}

The preconditioned matrix $\Htwo A$ that results from preconditioning by domain decomposition has the very appealing property that its eigenvalues can be bounded independently of the mesh size (parameterized by $L$ in our notation), and, if $Z$ is chosen correctly, of the number of subdomains ($N$ in our notation). This is fundamental in classical computing. The solver of choice would indeed  be the preconditioned conjugate gradient (PCG) method, and it converges fast if the eigenvalues of the preconditioned matrix are clustered away from zero (see \cite[Sec 6.11.3 and Sec. 9.2]{zbMATH01953444} and \cite[Lem. C.10]{zbMATH02113718}). We state below a fundamental result from the domain decomposition literature. 

\begin{theorem}[Spectrum of preconditioned operator \cite{zbMATH02113718}]
\label{th:DDtheory}
Let $\lambda_{\min}$ and $\lambda_{\max}$ be the smallest and largest eigenvalues of the preconditioned operator $\Htwo A$. We denote by $N_C$ the coloring constant of the partition into subdomains, \textit{i.e.}, the smallest number of colors needed to color
the subdomains so that two overlapping subdomains never share the same color. Assume that, either, 
\begin{itemize}
\item the columns of $Z$ are the coordinates in $V_L$ of the nodal basis functions associated with the coarse mesh induced by the subdomain partition (\textit{i.e.} the standard $\mathbb Q_1$ finite element basis for the coarse mesh in which each $\Omega^{(i)}$ is an element), 
\item or, the columns of $Z$ form a partition of unity subordinate to the subdomain decomposition, with each function supported in the corresponding subdomain and satisfying
\[
\|\nabla z_i\|_{L^\infty(\Omega^{(i)})}
\lesssim
\frac{1}{\delta}.
\]
\end{itemize}
Then, it holds that
\[
\lambda_{\min}\geq C \left(1 + \frac{\operatorname{diam}(\Omega^{(i)})}{\delta} \right)^{-1}  \left[ = C \left(1 + \frac{1}{\delta} \right)^{-1}\right] \text{ and } \lambda_{\max} \leq N_C + 1 \left[ = 5 \right],
\]
where $C$ is a constant that may depend on $N_C$ and $\rho$ (through the contrast ratio $\rho_{\max}/\rho_{\min}$) but does not depend on: the mesh size (parametrized by $L$), the number $N$ of subdomains, their diameter, or the overlap $\delta$.  
\end{theorem}
We have injected in brackets the formula that is particular to our choice of geometry. 
\begin{proof}
The lower bounds on the eigenvalues correspond to \cite[Th. 3.13]{zbMATH02113718} and \cite[Lem. 3.24]{zbMATH02113718}. We refer to there for the exact definitions of the coarse spaces.  It is to be noted that these results are originally for triangular meshes and piecewise linear finite element discretizations in the cited theorems. However, they also carry over to Cartesian meshes and our choice of finite elements. Indeed, the proofs refer back to results that are presented in \cite[Sec. App B1 and App B2]{zbMATH02113718} and whose original source is \cite[Sec. 3.4.1 and Sec 6.3.2]{zbMATH00635667}. Another difference is that the cited results are for Poisson with a constant $\rho$ equal to identity. This is accounted for by letting $C$ depend on $\rho$ under the assumption that $\rho$ satisfies \eqref{eq:assumption-rho}.

The upper bound on the eigenvalues is standard \cite[Lem 3.11]{zbMATH00635667}. The argument is that $\Htwo A$ can be rewritten as a sum of $N_C + 1$ $A$-orthogonal projections. 
\end{proof}

\begin{remark}[Other domain decomposition preconditioners]
A possible improvement is to apply the coarse correction multiplicatively by defining the hybrid Schwarz preconditioner
\begin{equation}
    \Hhyb := (I - P^{(0)}) \Hone (I - {P^{(0)}}^\top) +  Z (Z^\top A Z)^{-1} Z^\top,
\end{equation}
where $P^{(0)} = Z (Z^\top A Z)^{-1} Z^\top A$ is the $A$-orthogonal projection operator onto the coarse space. The spectrum of $\Hhyb A$ also satisfies the result from Theorem~\ref{th:DDtheory} and the upper bound can be improved to $N_C$.  

Another improvement allows to eliminate the dependency of the spectral bounds on $\rho$ in the case where $\rho$ is a constant scalar in each (non-overlapping) subdomain \cite[Th. 5]{zbMATH01592007}. This is done by considering the balanced Neumann-Neumann method \cite[Sec. 6.2]{zbMATH00635667}. Two major changes are that the local solves are for the Neumann problems on a non-overlapping partition, and that a partition of unity is introduced that depends on the distribution of $\rho$. We have not yet considered these improvements in the present work but deem them promising for future work.  
\end{remark}

The local solves ${A^{(i)}}^{-1}$ in the definition of the domain decomposition preconditioner can be replaced by approximations while retaining spectral bounds.  

\begin{theorem}[Inexact local solves]
We consider the two-level Additive Schwarz preconditioner 
$\Htwo = \sum_{i=1}^N {R_L^{(i)}}^\top {A^{(i)}}^{-1} R_L^{(i)} + Z (Z^\top A Z)^{-1} Z^\top$ defined by \eqref{eq:HASM2}. 
We also consider 
\begin{equation}
\Htwoap := \sum_{i=1}^N {R_L^{(i)}}^\top H^{(i)} R_L^{(i)} + Z (Z^\top A Z)^{-1} Z^\top,
\end{equation}
in which each matrix $H^{(i)}$ is symmetric positive definite. We assume that 

\begin{itemize}
\item the eigenvalues of $\Htwo A$ are in the interval $[\lambda_{\min},\lambda_{\max}]$, 
\item and, that the eigenvalues of $H^{(i)} A^{(i)}$ are in the interval $[\mu_{\min}, \mu_{\max}]$ for every index $i$. 
\end{itemize}
Then, the eigenvalues of the preconditioned system $\Htwoap A$ are in the interval \\ $[\lambda_{\min} \min (1, \mu_{\min}), \lambda_{\max} \max (1, \mu_{\max})]$.
\label{th:th_kappa_DD_BPX}
\end{theorem}

\begin{proof}
The assumptions on the spectra of  $\Htwo A$ and  $H^{(i)} A^{(i)}$ can be rewritten as bounds of well-chosen Rayleigh quotient: 
\[
\lambda_{\min}\, x^\top A x \leq  x^\top (A \Htwo A) x \leq  \lambda_{\max}\, x^\top  A x, \quad \forall x, 
\]
and 
\[
\mu_{\min}\, {x^{(i)} }^\top  {A^{(i)}}^{-1} x^{(i)} \leq    {x^{(i) }}^\top  {H^{(i)}}  x^{(i)}  \leq  \mu_{\max}\, {x^{(i) }}^\top  {A^{(i)}}^{-1} x^{(i)},  \quad \forall x^{(i)}. 
\]
Combining these gives the result in the form\\ $\lambda_{\min} \min (1, \mu_{\min})\, x^\top A x \leq  x^\top A \Htwoap A x \leq    \lambda_{\max}\, \max (1, \mu_{\max})x^\top  A x, \quad \forall x $ 

\end{proof}
We note that we could also consider an inexact coarse solve $(Z^\top A Z)^{-1}$ with no extra difficulty.

All above results are for the spectrum of the preconditioned operator. Since the preconditioned matrix is non-symmetric, these do not imply bounds for the condition number of the preconditioned matrix\footnote{Note that in the DD literature, it is often said that the condition number is bounded. However this refers to the condition number for the norm induced by $A$, not the spectral condition number.}. In order to obtain condition number bounds, we use again an idea from \cite{Deiml_2025} and split the preconditioner to symmetrize the preconditioned operator.  

\subsection{Split form of the preconditioned operator}
\label{sec:splitprec}

Let's assume that, for any $i=1,\dots,N$, the inexact local solvers $H^{(i)}$ can be factorized as $H^{(i)} =: F^{(i)} {F^{(i)}}^\top$ and that the coarse solve can be factorized as $ (Z^\top A Z)^{-1} =: F^{(0)} {F^{(0)}}^\top $. We do not require the matrices $F^{(i)}$ and $F^{(0)}$ to be square. Then $\Htwoap$ rewrites
\begin{equation}
\Htwoap = \sum_{i=1}^N {R_L^{(i)}}^\top F^{(i)} {F^{(i)}}^\top R_L^{(i)} + Z F^{(0)} {F^{(0)}}^\top Z^\top = \tilde{F} \tilde{F}^\top,
\label{eq:decomposition_preconditioner}
\end{equation}
with
\begin{equation}
    \tilde{F} := \left({R_L^{(1)}}^\top F^{(1)} | ~\dots~| {R_L^{(i)}}^\top F^{(i)} | \dots~ | {R_L^{(N)}}^\top F^{(N)} | Z F^{(0)} \right).
\end{equation}

By this notation, we mean that we have concatenated all matrices of the form ${R_L^{(i)}}^\top F^{(i)}$ horizontally with the matrix $Z F^{(0)}$. The matrix $\tilde{F}$ is a rectangular matrix. It has $\mathcal N$ rows.

Based on the decomposition of the stiffness matrix \cref{eq:decomposition_stiffness_matrix} and the decomposition of the preconditioner \cref{eq:decomposition_preconditioner}, we arrive at a symmetric preconditioned system 
\begin{equation}
    \tilde{F}^\top C_L^\top (D_{\rho} \otimes Id_{2^d}) C_L \tilde{F} y = \tilde{F}^\top b , \text{ with } x=\tilde{F} y.
    \label{eq:prec_system}
\end{equation}
This is the system that is solved \textit{in fine} by our methodology. The symmetric preconditioned operator is $\tilde{F}^\top C_L^\top (D_{\rho} \otimes Id_{2^d}) C_L \tilde{F}$. Its order is larger than the order $\mathcal N$ of the original preconditioned matrix $\Htwoap A$. Its rank is $\mathcal N$. The non-zero eigenvalues of $\Tilde{F}^\top C_L^\top (D_{\rho} \otimes Id_{2^d}) C_L \tilde{F}$ are the eigenvalues of $\Htwoap A$ so that 
\[
\kappa(\tilde{F}^\top C_L^\top (D_{\rho} \otimes Id_{2^d}) C_L \tilde{F}) = \frac{\lambda_{\max}(\Htwoap A)}{\lambda_{\min}(\Htwoap A)},
\]
where $\lambda_{\max}(\Htwoap A)$ and $\lambda_{\min}(\Htwoap A)$ are the extreme eigenvalues of $\Htwoap A$. These can be bounded by Theorems~\ref{th:DDtheory} and~\ref{th:th_kappa_DD_BPX}. 
The solution $y$ to \eqref{eq:prec_system} is defined only up to an element in the kernel of $\tilde{F}$ but this does not modify the corresponding solution $x = \tilde{F} y$ of the original system.

In Section~\ref{sec:quantum}, a block-encoding of the matrix $\tilde{F}^\top C_L^\top (D_{\rho} \otimes Id_{2^d}) C_L \tilde{F}$ is proposed. Because the matrices $C_L$ and $\tilde{F}$ are not well-conditioned it is chosen to avoid implementing it as the product of block-encodings. Instead, the product 
\[
C_L \tilde{F} = \left(C_L {R_L^{(1)}}^\top F^{(1)} | ~\dots~| C_L{R_L^{(i)}}^\top F^{(i)} | \dots~ | C_L{R_L^{(N)}}^\top F^{(N)} | C_L Z F^{(0)} \right) 
\]
is reformulated so that each of its blocks is block-encoded at once. 

For each $i = 1, \dots, N$ the local discontinuous Galerkin space is introduced as 
\[
Q_L^{(i)} := \left\{u_L^{(i)} \in L^2(\Omega^{(i)});\, {u_L^{(i)}}_{|K} \in \mathbb Q_1(K), \ \forall K \in \mathcal{T}_L^{(i)} \right\},
\]
\textit{i.e.}, the local counterpart of $Q_L$ from equation~\eqref{eq:QL}. The vector-valued space ${Q_L^{(i)}}^d := Q_L^{(i)} \otimes \mathbb{R}^d$ is also introduced. These spaces are equipped, respectively, with with an $L^2$-orthonormal basis consisting of $2^d$ functions
$\psi_{K,1}, \dots, \psi_{K,2^d}$ on each element $K \in \mathcal{T}_L^{(i)}$, and with the basis functions $\psi_{K,k} \otimes e_s $ where $e_s$ is the $s$-th canonical basis vector in $\mathbb R^{d}$.  We also let $C_L^{(i)}$ denote the matrix representation of the (elementwise) gradient operator
\[
\nabla : V_L^{(i)} \to {Q_L^{(i)}}^d.
\]
There is a natural embedding of ${Q_L^{(i)}}^d$ into $Q_L^{d}$. It induces an extension operator from ${Q_L^{(i)}}^d$ into $Q_L^{d}$ that is represented by a matrix which we denote by ${\mathcal{R}_L^{(i)}}^\top$. Each term $C_L {R_L^{(i)}}^\top$ can then be rewritten as 
\begin{equation}
C_L {R_L^{(i)}}^\top = {\mathcal{R}_L^{(i)}}^\top C_L^{(i)},
\label{eq:switch_gradient}
\end{equation}
and both are equal to $C_L(:,\omega_i)$, \textit{i.e.}, the selection of columns in $C_L$ that corresponds to degrees of freedom of $V_L$ that are in the interior of $\Omega^{(i)}$. This can be verified by considering the operators represented by these matrices: for a function in $ V_L^{(i)}$, it is equivalent to first prolongate it by zero to $\Omega$ and then compute the element-wise (discontinuous) gradient, or to first compute the element-wise (discontinuous) gradient and then prolongate it by zero to $\Omega$ as summarized in the diagram below 
\[
\begin{tikzcd}[row sep=normal, column sep=huge]
V_L^{(i)} \arrow[r,hook, "{ {R_L^{(i)}}^\top }"] \arrow[d, "C_L^{(i)}"'] & V_L \arrow[d, "C_L"] \\
{Q_L^{(i)}}^d \arrow[r,hook, "{ {\mathcal{R}_L^{(i)}}^\top }"'] & Q_L^d
\end{tikzcd}
\]

We arrive at
\begin{equation}
C_L \tilde{F} = \left( {\mathcal{R}_L^{(1)}}^\top C_L^{(1)} F^{(1)}  |~\dots~| {\mathcal{R}_L^{(i)}}^\top C_L^{(i)} F^{(i)} | \dots~| {\mathcal{R}_L^{(N)}}^\top C_L^{(N)} F^{(N)} |C_L Z F^{(0)} \right),
\label{eq:precond_gradient}
\end{equation}
and propose to block-encode this matrix in the next section, as a step towards block-encoding $\tilde{F}^\top C_L^\top (D_{\rho} \otimes Id_{2^d}) C_L \tilde{F}$ and proposing a QLSS for  \eqref{eq:prec_system}. We have not transformed the term $C_L Z F^{(0)}$.

\subsection{Advantages of domain decomposition}
\label{sec:DD-adv}
Now that we have introduced domain decomposition, also known as DD, we present in more detail some of its strengths that could carry over to the quantum formalism.

\begin{itemize}
    \item As shown in Theorem~\ref{th:DDtheory}, preconditioning a linear system by domain decomposition allows to improve significantly the eigenvalue distribution and to make it independent of some parameters such as the mesh size. Theorem~\ref{th:DDtheory} is only one example of such a result and we want to stress that domain decomposition preconditioners have been developed for a very wide range of PDEs. For this reason, the feasibility of quantum domain decomposition that is achieved in this article paves the way to solving many more problems than the Poisson problem on a hyperrectangle. 
    \item Another advantage of domain decomposition is that it can be viewed as an abstract framework into which various local and coarse solvers can be plugged. This is by now quite standard in classical computing (in particular \textit{via}, the PETSc library \cite{petsc-user-ref}) and such versatility could be carried over to quantum implementations. 
    \item One particular strength of domain decomposition is to address the solution of PDEs with heterogeneous coefficients. The level of difficulty depends on whether the heterogeneities are across the subdomains boundaries, inside the subdomains or whether they cross the boundaries. A good choice of coarse space is instrumental to ensure convergence with respect to the coefficient distribution \cite{zbMATH05172548,zbMATH01592007,zbMATH05651315,spillane_geneo_2025,zbMATH06303461,zbMATH02113718}. 
    \item The coarse space formalism mentioned above may turn out not to be the best suited for quantum implementation. However, it remains from the corresponding analysis that only a small number of eigenvalues are isolated and responsible for an unfavourable condition number. These could be filtered away by other techniques. 
    \item Domain decomposition is also a suitable framework for addressing multiphysics formulations. 
    \item Although it is not the point of view adopted in this article, domain decomposition algorithms could be useful to parallelize computations over multiple quantum processing units. 
    \item A recent contribution by \cite{busaleh2026mitigatingbarrenplateausdomain} is to alleviate Barren Plateau-like effects in variational quantum algorithms by replacing a global cost function by localized ones by domain decomposition techniques. 
\item A final group of ideas concerns the reformulation of $Ax=b$ through substructuring \cite[Chap. 4, 5, 6]{zbMATH02113718}. After partitioning the degrees of freedom into interior and overlapped variables $x = (x_I, x_\Gamma)$, the interior unknowns can be eliminated via the local subdomain operators $A^{(i)}$. The global system can then be reduced to an equivalent Schur complement problem posed only on the interface variables:
\[
S x_\Gamma = b_\Gamma.
\]
In this formulation, the unknown that must be represented or accessed by a quantum algorithm is no longer the full vector $x$, but only its interface component $x_\Gamma$. The interior degrees of freedom can be reconstructed classically through independent local subdomain solves, so that the quantum state preparation is restricted to a reduced interface space.
\item Elimination of the interior degrees of freedom also changes the structure of the data that must be encoded in quantum routines for linear system solving. Indeed, after static condensation, the effective right-hand side for the reduced system is supported only on the overlapping degrees of freedom. As a consequence, both the state preparation for $b$ and the representation of residuals in iterative schemes can be restricted to a smaller overlap-supported set of components. This suggests a reduction in the number of amplitudes required to encode problem data for quantum linear algebra subroutines.
\item Finally, this structure could be exploited within the hybrid QLSS proposed in \cite{Koska:2025ucz}. The algorithm is an iterative refinement strategy inspired by mixed-precision numerical linear algebra techniques.  A first quantum solution is computed using QSVT in low precision, and then refined in higher precision until a satisfactory accuracy is achieved. Since the residual driving the refinement inherits the overlap-supported structure induced by domain decomposition, the QSVT correction step can be reformulated to act on reduced interface residuals rather than full-system vectors. This suggests that each quantum linear system call within the iterative refinement loop may be applied to lower-dimensional data, potentially reducing quantum resources per iteration while preserving the convergence properties of the original scheme.
\end{itemize}

\subsection{Numerical illustration of domain decomposition}
The purpose of this section is to illustrate the domain decomposition knowledge that we have introduced, and to test the proposed preconditioner before carrying it over to the quantum formalism. We consider the two-dimensional Poisson problem ($d=2$) with $f=1$ and $\rho=I_d$ over the whole domain $\Omega$. The reported results are for the ratio between the largest and smallest eigenvalue of the (un)preconditioned operator. This is also the condition number of the split preconditioned operator so when we speak of condition number in this section we refer to ${\lambda_{\max}}/{\lambda_{\min}}$. 

In \cref{tab:classical_dd_partition}, the number of subdomains varies in both directions. The mesh size is $h=2^{-4}$ (\textit{i.e.} L = 4), and the overlap is $\delta = 2^{-4}$. Two-level Additive Schwarz preconditioners $\Htwo$ and $\Htwoap$ with exact and inexact local solves are compared to two-level hybrid Schwarz preconditioners $\Hhyb$ and $\Hhybap$  with exact and inexact local solves. For the inexact local solves, we replace ${A^{(i)}}^{-1}$ by the standard BPX preconditioner \cite{zbMATH04197323}. Its definition is deferred to \cref{sec:quantumDDBPX}. We first observe that all preconditioners drastically improve the condition number: from $1499$ to at most $76$ when there are $6 \times 6$ subdomains.  In agreement with Theorem~\ref{th:DDtheory}, the condition numbers of the preconditioned operators do not depend significantly on the number of subdomains. This is particularly true for the hybrid two-level preconditioner $\Hhyb$ with an almost constant condition number close to $25$. For $\Htwo$, the increase in condition number does not contradict the theory, it indicates that the asymptotic regime has not been reached. The preconditioners with inexact local solves perform only slightly worse than their exact counterparts which justifies their use.

\begin{table}[htbp]
\caption{Ratio of the largest and smallest eigenvalues of the (un)preconditioned matrix. The number of subdomains varies in both directions. The mesh size and overlap are fixed $h=\delta=2^{-4}$.}
\label{tab:classical_dd_partition}
\begin{center}
\begin{tabular}{| c |c | c | c | c | c |}
 \hline
 Partition & \multicolumn{5}{|c|}{${\lambda_{\max}}/{\lambda_{\min}}$} \\[0.5ex] \hline
 & & \multicolumn{2}{|c|}{Exact local solves} & \multicolumn{2}{|c|}{Inexact local solves} \\[0.5ex] \hline
 & unpreconditioned & $\Htwo$ & $\Hhyb$& $\Htwoap$  & $\Hhybap$  \\ [0.5ex]
 \hline
 $N = 3\times3$ & $392.1$ & $34.5$ & $24.9$ &  $44.2$ &$27.5$  \\
 \hline
$N = 4\times 4$ & $681.5$ & $40.1$ &$25.6$ & $57.9$ &  $28.6$ \\
 \hline
$N = 5\times 5$ & $1050.3$ & $43.2$& $26.0$ & $68.3$  & $29.1$ \\
 \hline
$N = 6\times 6$ & $1498.6$ & $45.1$& $26.2$ & $75.9$  & $29.4$ \\
 \hline
\end{tabular}
\end{center}
\end{table}

We consider a banded domain for \cref{tab:classical_dd_discretization} and \cref{tab:classical_dd_delta}: subdomains are only added along direction $1$ (\textit{i.e.} $N_2=1$). In this case, a classical result in domain decomposition is that the coarse correction is not necessary because all subdomains are connected through the Dirichlet boundary condition. \cref{tab:classical_dd_discretization} illustrates the influence of the mesh size for the one-level Additive Schwarz preconditioners $H_1$ (exact) and $\tilde H_1$ (inexact). The overlap is kept constant at $\delta = 2^{-3}$ and the number of subdomains is $N_1=8$. Using exact local solves, the eigenvalue ratio does not depend on the mesh size. This is in agreement with Theorem~\ref{th:DDtheory}. In the case of inexact local solves, the eigenvalue ratio slightly increases from $4.2$ to $8.5$ when the mesh size decreases from $2^{-3}$ to $2^{-5}$.

\begin{table}[htbp]
\caption{Ratio of the largest and smallest eigenvalues of the (un)preconditioned matrix. The mesh size varies from $2^{-3}$ to $2^{-5}$. The number of subdomains is $N_1=8$ and the overlap is $\delta=2^{-3}$.}\label{tab:classical_dd_discretization}
\begin{center}
\begin{tabular}{| c |c | c | c |}
 \hline
 Discretization & \multicolumn{3}{|c|}{${\lambda_{\max}}/{\lambda_{\min}}$} \\[0.5ex] \hline
 & unpreconditioned & $\Hone$ & $\Honeap$ \\ [0.5ex] \hline
$L=3$ & 25.3 & 4.1 & 4.2 \\ \hline
$L=4$ & 101.2 & 4.2 & 6.9 \\ \hline
$L=5$ & 404.7 & 4.2 & 8.5 \\\hline
\end{tabular}
\end{center}
\end{table}

In \cref{tab:classical_dd_delta}, we study the influence of the overlap $\delta$ for the one-level Additive Schwarz preconditioners. The number of subdomains and the mesh size are kept constant ($N_1=8$ and $h=2^{-5}$ corresponding to $L=5$). For both the exact and inexact local solves, the eigenvalue ratio decreases when the overlap increases. Once more, this corresponds to the prediction of Theorem~\ref{th:DDtheory}.

\begin{table}[htbp]
\caption{Ratio of the largest and smallest eigenvalues of the (un)preconditioned matrix. The overlap varies. The number of subdomains is $N_1=8$ and the mesh size is $h=2^{-5}$.}\label{tab:classical_dd_delta}
\begin{center}
\begin{tabular}{| c |c | c | c |}
 \hline
 Overlap & \multicolumn{3}{|c|}{${\lambda_{\max}}/{\lambda_{\min}}$} \\[0.5ex] \hline
 & unpreconditioned & $\Hone$ & $\Honeap$ \\ [0.5ex] \hline
$\delta=2^{-5}$ & 407.88 & 12.24 & 19.53 \\ \hline
$\delta=2^{-4} $ & 406.99 & 6.79 & 12.93 \\ \hline
$\delta=2^{-3}$ & 404.65 & 4.18 & 8.46 \\\hline
$\delta=2^{-2}$ & 395.48 & 3.20 & 7.71 \\\hline
\end{tabular}
\end{center}
\end{table}

\section{Quantum algorithm}
\label{sec:quantum}
We first introduce the block-encoding framework and state the complexity of a particular quantum linear system solver. Then, we discuss the block-encoding of the symmetric preconditioned matrix $\tilde{F}^\top A \tilde{F}$, which was introduced in the previous section by symmetrizing $\Htwoap A$. We also discuss the implementation of the block-encoding and the cost of inverting the linear system (in the quantum sense). 

\subsection{Preliminaries}
The block-encoding framework allows to embed a general matrix inside a larger unitary matrix \cite{Gily_n_2019}. We use a general definition \cite{Gily_n_2019, Deiml_2025} also known as projected unitary encoding.
\begin{definition}[Block-encoding -- adapted from{\cite[Definition 3.2 and Definition 3.1]{Deiml_2025}}]
Let $M_1,M_2,p \in \mathbb{N}$ such that $M_1, M_2 \leq P:=2^p$. Let $\AA \in \mathbb{C}^{M_1\times M_2}$ and $\alpha \geq \| \AA \|$. Let $\Pi_{M_1}:\mathbb{C}^{P} \rightarrow \mathbb{C}^{M_1}$ and $\Pi_{M_2}:\mathbb{C}^{P} \rightarrow \mathbb{C}^{M_2}$ be projections\footnote{In this context a projection is a matrix with orthonormal rows.} implemented using $C_{\Pi_{M_1}}NOT$ and $C_{\Pi_{M_2}}NOT$ gates. The $C_{\Pi_{M_1/M_2}}NOT$ gate performs the operation
\begin{equation*}
    \Pi_{M_1/M_2}^\dag\Pi_{M_1/M_2} \otimes X + (I-\Pi_{M_1/M_2}^\dag\Pi_{M_1/M_2} ) \otimes I.
\end{equation*}

For $U_{\AA} \in \mathbb{C}^{P \times P}$, $\left( \alpha, U_{\AA}, C_{\Pi_{M_1}}NOT, C_{\Pi_{M_2}}NOT \right)$ is a block-encoding of $\AA$ if 
\begin{equation*}
    \AA = \alpha \Pi_{M_1} U_\AA \Pi_{M_2}^{\dag}.
\end{equation*}
The constant $\alpha$ is the normalization factor. It will be denoted by ${\alpha}(U_{\AA})$ to make explicit the block-encoding to which it corresponds. We also define the subnormalization factor $\tilde{\alpha}(U_{\AA}) := \frac{\alpha(U_{\AA})}{\| \AA \|}$.
\end{definition}

 Considering a general, well-posed, system of the form $\AA x = b$, quantum linear system solvers output a quantum state close to $\ket{x} \propto \AA^{+} \ket{b}$. We recall the complexity of a QSVT-based QLSS from  {\cite[Theorem 33]{PhysRevA.104.032422}} and {\cite[Theorem 19.3]{lin2026lecturenotes}}. 
\begin{lemma}[QLSS based on QSVT]
Let $\left( \alpha, U_\AA, C_{\Pi_{M_1}}NOT, C_{\Pi_{M_2}}NOT \right)$ be a block-encoding of $\AA$, and let $U_b$ be an oracle preparing $\ket{b}$. There exists a quantum algorithm that outputs a quantum state $\epsilon$-close to the quantum state $\ket{x} = \frac{\AA^{+} \ket{b}}{\|\AA^{+} \ket{b}\|}$ using $\bigO{\tilde{\alpha}(U_{\AA}) \kappa^2(\AA) \log\left(\frac{\kappa(\AA)}{\epsilon}\right)}$ queries to $U_\AA$ and $\bigO{\kappa(\AA)}$ queries to $U_b$.
\label{lem:qlss_qsvt}
\end{lemma}
The complexity can be further improved to almost linear in $\kappa(\mathcal A)$ using variable time amplitude amplification \cite{zbMATH06070940}. In the cited references, $\AA$ is non-singular and the lemma is written with $\AA^{-1}$ instead of the pseudo-inverse $\AA^+$. However pseudo-inversion is not a difficulty for QSVT and was already considered in the original article \cite{Gily_n_2019}. 

From this result, we observe that QSVT is efficient if $\AA$ is well conditioned, and if there exists a block-encoding of $\AA$ with a small subnormalization factor. To improve the condition number, we precondition by domain decomposition as proposed in Section~\ref{sec:DD}. It is then the condition number of the split preconditioned operator 
\[
\tilde{F}^\top A \tilde{F} = \tilde{F}^\top C_L^\top (D_{\rho} \otimes Id_{2^d}) C_L \tilde{F}.
\]
which plays a role and we have already analyzed it. We are left with the analysis of the subnormalization factor $\tilde{\alpha}(U_{\tilde{F}^\top A \tilde{F}})$.
We will propose a particular choice of local solvers ${F^{(i)}}^\top F^{(i)}$ in the domain decomposition preconditioner for which $\tilde{\alpha}(U_{\tilde{F}^\top A \tilde{F}})$ remains bounded independently of the problem parameters.

\subsection{Block-encoding of the preconditioned system}
The following theorem presents a block-encoding of the preconditioned system provided access to a block-encoding of the coefficient $D_\rho$ (following \textit{e.g}, \cite[Proposition 3.4]{Deiml_2025}), as well as to block-encodings of the local matrices $C_L^{(i)} F^{(i)}$ and to a block encoding of the coarse correction $C_L Z F^{(0)}$. We recall that these matrices form the blocks of $C_L \tilde{F}$ in \eqref{eq:precond_gradient}, the definition of which we state again for clarity:
\[
C_L \tilde{F} = \left( {\mathcal{R}_L^{(1)}}^\top C_L^{(1)} F^{(1)}  |~\dots~| {\mathcal{R}_L^{(i)}}^\top C_L^{(i)} F^{(i)} | \dots~| {\mathcal{R}_L^{(N)}}^\top C_L^{(N)} F^{(N)} | C_L Z F^{(0)} \right).
\]

\begin{theorem}
Let $U_{D_\rho}$ and $U_{C_L Z F^{(0)}}$ be block-encodings of $D_\rho$ and $C_L Z F^{(0)}$. For any $i=1, \dots, N$, let  $U_{C_L^{(i)} F^{(i)}}$ be a block-encoding of $C_L^{(i)} F^{(i)}$. There exists a block-encoding $U_{\tilde{F}^\top A \tilde{F}}$ of $\tilde{F}^\top A \tilde{F}$ with normalization factor
\begin{equation}
\alpha (U_{\tilde{F}^\top A \tilde{F}}) = \alpha(U_{D_\rho}) \left( \sum_{i=1}^{N} \alpha\left(U_{C_L^{(i)} F^{(i)}}\right)^2 + \alpha \left( U_{C_L Z F^{(0)}} \right)^2 \right),
\label{eq:normalization_BE_DD}
\end{equation}
and subnormalization factor
\begin{equation}
\tilde{\alpha} (U_{\tilde{F}^\top A \tilde{F}}) \leq \kappa(D_\rho) \tilde{\alpha}(U_{D_\rho}) \left(\sum_{i=1}^{N} \tilde{\alpha}\left(U_{C_L^{(i)} F^{(i)}}\right)^2 + \tilde{\alpha}\left( U_{C_L Z F^{(0)}} \right)^2 \right).
\label{eq:subnormalization_BE_DD}
\end{equation}
\label{thm:th_BE_DD}
\end{theorem}

\begin{proof}
For any  $1\leq i \leq N$, $\mathcal{R}_L^{(i)}$ is orthogonal so
$\kappa(\mathcal{R}_L^{(i)}) = 1$ and there exists a block-encoding $U_{\mathcal{R}_L^{(i)}}$ of $\mathcal{R}_L^{(i)}$ such that $\alpha(U_{\mathcal{R}_L^{(i)}})=\tilde\alpha(U_{\mathcal{R}_L^{(i)}}) =1$. Then, applying \cite[(3.1.e)]{Deiml_2025} followed by \cite[(3.1.d)]{Deiml_2025}, the normalization factor of the block-encoding of the block matrix $C_L \tilde{F}$ is
\[
\begin{aligned}
\alpha (U_{C_L \tilde{F}}) &= \left( \sum_{i=1}^{N} \alpha\left(U_{{\mathcal{R}_L^{(i)}}^\top C_L^{(i)} F^{(i)}}\right)^2 + \alpha \left(U_{C_L Z F^{(0)}} \right)^2 \right)^{1/2},
\\ &= \left( \sum_{i=1}^{N} \alpha\left(U_{C_L^{(i)} F^{(i)}}\right)^2 + \alpha \left(U_{C_L Z F^{(0)}} \right)^2   \right)^{1/2}.
\end{aligned}
\]

Using the same arguments the subnormalization factor is bounded by
\[
\tilde{\alpha} \left(U_{C_L \tilde{F}}\right) \leq \left( \sum_{i=1}^{N} \tilde{\alpha}\left(U_{  C_L^{(i)} F^{(i)}}\right)^2 + \tilde{\alpha}\left( U_{C_L Z F^{(0)}} \right)^2 \right)^{1/2}.
\]

By \cite[(3.1.d)]{Deiml_2025}, the normalization factor of the preconditioned system is 
\[
\alpha (U_{\tilde{F}^\top A \tilde{F}}) = \alpha(U_{D_\rho}) \alpha(U_{C_L \tilde{F}})^2 = \alpha(U_{D_\rho}) \left( \sum_{i=1}^{N} \alpha(U_{C_L^{(i)} F^{(i)}})^2 + \alpha \left(U_{C_L Z F^{(0)}} \right)^2 \right).
\]
The subnormalization factor of the preconditioned system is proved by following \cite[Proof of Theorem 6.2]{Deiml_2025}

\[
\| \tilde{F}^\top A \tilde{F} \| = \|\tilde{F}^\top C_L^\top \left( D_{\rho} \otimes I_{2^d} \right) C_L \tilde{F} \| = \|\left( D_{\rho}^{1/2} \otimes I_{2^d} \right) C_L \tilde{F} \|^2 \geq \frac{\| D_\rho^{1/2} \|^2 \| C_L \tilde{F} \|^2}{ \kappa(D_\rho)}.
\]
The last inequality follows from $D_{\rho}^{1/2}$ being invertible and 
\[
\begin{aligned}
\| D_\rho^{1/2} \|& \| C_L \tilde{F} \| = \| D_\rho^{1/2} \| \left\| \left( ( {D_\rho^{1/2}}^+ D_\rho^{1/2})\otimes I_{2^d}\right) C_L \tilde{F} \right\| 
\\& \leq \| D_\rho^{1/2} \| \| {D_\rho^{1/2}}^+ \| \left\|\left(D_\rho^{1/2} \otimes I_{2^d} \right) C_L \tilde{F} \right\| 
 = \kappa(D_\rho^{1/2}) \left\| \left(D_\rho^{1/2} \otimes I_{2^d} \right) C_L \tilde{F} \right\|.
\end{aligned}
\]
The subnormalization factor is then bounded by
\[
\begin{aligned}
\tilde{\alpha} (U_{\tilde{F}^\top A \tilde{F}})  = \frac{\alpha (U_{\tilde{F}^\top A \tilde{F}})}{\| \tilde{F}^\top A \tilde{F} \|} & \leq \kappa(D_\rho) \tilde{\alpha}(U_{D_\rho}) \tilde{\alpha}(U_{C_L \tilde{F}})^2 \\
& \leq \kappa(D_\rho) \tilde{\alpha}(U_{D_\rho}) \left( \sum_{i=1}^{N} \tilde{\alpha}\left(U_{C_L^{(i)} F^{(i)}}\right)^2 + \tilde{\alpha}\left( U_{C_L Z F^{(0)}} \right)^2  \right).
\end{aligned}
\]

\end{proof}

In order to propose a full quantum domain decomposition result, we discuss implementation choices in Section~\ref{sec:impRitilde}, particularly for the block encoding of the matrices $\mathcal{R}_L^{(i)}$. Then, in Section~\ref{sec:quantumDDBPX}, we set the local solvers $F^{(i)}{F^{(i)}}^\top$ to be the BPX approximation of ${A^{(i)}}^{-1}$, and give our final result in \cref{thm:QLSS_DD_BPX}.

\subsection{Implementation}
\label{sec:impRitilde}
Up until now, we have not considered the concrete implementation of the block-encodings. An efficient circuit that block-encodes the preconditioned system $U_{\tilde{F}^\top A \tilde{F}}$ is necessary in order to ensure the efficiency of the QLSS. The same remark applies to the state-preparation oracles. 

We first detail the different registers that are involved in the block-encodings. There are five: 
\begin{itemize}
\item Registers $\ket{j}$, $\ket{k}$ and $\ket{s}$ are the same as in \cite{Deiml_2025}. Register $\ket{j}$ encodes the mesh elements or the degrees of freedom of the local Lagrange finite element spaces. Based on the tensor product structure of the Cartesian mesh and of the Lagrange spaces, it consists of $d$ sub-registers $\ket{j_1} \cdots \ket{j_d}$, each containing $L$ qubits. Register $\ket{k}$ is associated with the $2^d$ degrees of freedom per mesh element of the local Discontinuous Galerkin spaces, so it contains $d$ qubits $\ket{k_1} \cdots \ket{k_d}$. Register $\ket{s}$, which consists of $\log_2(d)$ qubits, encodes the $d$ components of the gradient. 
\item We introduce a register $\ket{i}$ associated with the $N$ subdomains. It is composed of $d$ registers $\ket{i_1} \cdots \ket{i_d}$ (see Remark~\ref{rem:i-multiindex}). For $1 \leq s \leq d $, register $\ket{i_s}$ contains $n_s= log_2(N_s)$ qubits (where $N_s$ is the number of subdomains in direction $s$).
\item Finally, we introduce a new register $\ket{\tilde{i}}$ which has the same structure as $\ket{i}$ in order to perform quantum domain decomposition. 
\end{itemize}
\medskip

\paragraph{Local to global Prolongation operators}
We now propose a block-encoding of the prolongation operators ${\mathcal{R}^{(i)}_L}^\top : {Q_L^{(i)}}^d \hookrightarrow Q_L^{d}$ defined in Section~\ref{sec:splitprec}. Each matrix ${\mathcal{R}^{(i)}_L}^\top$ is of order $d 2^{(L+1)d} \prod_{s=1}^d (N_s - 2 \delta (N_s-1)) \times d 2^{(L+1)d}$. Having split the domain $\Omega$ into an overlapping Cartesian grid of subdomains, the prolongation operators display a tensor-product structure,
\begin{equation}
    {\mathcal{R}^{(i)}_L}^\top = I_d \otimes {\mathcal{R}^{(i_1)}_{L,1D}}^\top \otimes \cdots \otimes {\mathcal{R}^{(i_d)}_{L,1D}}^\top,
\end{equation}
where each $\mathcal{R}^{(\text{ind}_s)}_{L,1D}$ performs the restriction to local elements for the $s$-th coordinate. Therefore, we can focus on the implementation of one of the ${\mathcal{R}^{(\text{ind}_s)}_{L,1D}}^\top$ with $1 \leq \text{ind}_s \leq N_s$.
The index of a Discontinuous Galerkin degree of freedom encoded by $\ket{k_s}\ket{j_s}\ket{i_s}$ is $k_s + 2 j_s +2^{L+1} i_s - 2^{L+2} \delta i_s$. The term $-2^{L+2} \delta i_s$ accounts for the overlap. Quantum arithmetic based on the Fourier transform \cite{draper2000additionquantumcomputer} enables the implementation of the operator $U_{{\mathcal{R}^{(\text{ind}_s)}_{L,1D}}^\top}$
\begin{equation}
    \ket{k_s,j_s,i_s}\ket{\tilde{i}_s} \rightarrow \ket{k_s + 2j_s +2^{L+1}i_s - 2^{L+2}\tilde{i}_s \delta}\ket{\tilde{i}_s}.
\end{equation}

The row index encoded by $\ket{k_s,j_s,i_s,\tilde{i}_s}$ is the integer $p:=k_s + 2j_s +2^{L+1}i_s + 2^{n_s+L+1}\tilde{i}_s$.
The $C_{\Pi_{row}} \text{NOT}$ gate flips the ancilla qubit when the index of the row satisfies $p \leq M_s := N_s 2^{L+1} - 2^{L+2} \delta(N_s -1) + N_s 2^{L+1} (\text{ind}_s - 1)$ and $p \geq N_s 2^{L+1} (\text{ind}_s - 1)$. The number of rows selected by $C_{\Pi_{row}} \text{NOT}$ corresponds to the total number of degrees of freedom for the Discontinuous Galerkin space in direction $s$, \textit{i.e.} $N_s 2^{L+1} - 2^{L+2} \delta(N_s -1)$. Each of the two inequality tests can also be performed by QFT-based quantum arithmetic with one ancilla qubit \cite{Yuan_2023,10.5555/2011517.2011525}. For the inequality $p \leq M_s$, the following operations are carried out in the Fourier spaces
\begin{subequations}
\begin{align}
    \ket{k_s,j_s,i_s,\tilde{i}_s}\ket{0} &\rightarrow \ket{p - M_s - 1}_{p \geq M_s+1}\ket{0} + \ket{p - M_s -1 + N_s^2 2^{L+1}}_{p < M_s+1}\ket{1} \\
    & \rightarrow \ket{k_s,j_s,i_s,\tilde{i}_s}\ket{p < M_s+1}.
\end{align}
\end{subequations}
Similar operations are performed for the inequality $p \geq N_s 2^{L+1} (\text{ind}_s - 1)$.
The $C_{\Pi_{col}}\text{NOT}$ gate flips the ancilla qubit when $\ket{\tilde{i}_s} = \ket{i_s} = \text{ind}_s-1$. This equality test can be performed using quantum arithmetic based on multi-controlled gates. This gate selects the $2^{L+1}$ columns.

\paragraph{Preparation of the right hand side}
The efficiency of the quantum linear system solver also relies on the efficiency of the state preparation. We address the preparation of the preconditioned right hand side 
\[
{\tilde{F}}^\top b = \begin{pmatrix}
{F^{(1)}}^\top R_L^{(1)} b \\
\vdots
\\
{F^{(N)}}^\top R_L^{(N)} b \\
{F^{(0)}}^\top Z^\top b
\end{pmatrix}.
\]

We assume that we already know how to efficiently encode local components $\displaystyle \ket{{F^{(i)}}^\top R_L^{(i)} b } := \frac{{F^{(i)}}^\top R_L^{(i)} b }{\|{F^{(i)}}^\top R_L^{(i)} b \|}$ and the coarse component $\displaystyle \ket{{F^{(0)}}^\top Z^\top b } := \frac{{F^{(0)}}^\top Z^\top b }{\|{F^{(0)}}^\top Z^\top b \|}$ using $m$ qubits. Supposing access to all of the norms, we can then prepare the $(n+1)$-qubit state $\ket{0}_{n+1} \rightarrow \frac{\|{F^{(0)}}^\top Z^\top b\|}{\|{\tilde{F}}^\top b\|} \ket{0}_{n+1} + \sum_{i=1}^N \frac{\|{F^{(i)}}^\top R_L^{(i)} b\|}{\|{\tilde{F}}^\top b\|} \ket{i}_{n+1}$ with a cost at most linear in $N=2^n$. Controlling the state preparation of $\ket{{F^{(i)}}^\top R_L^{(i)} b}$ and $\ket{{F^{(0)}}^\top Z^\top b}$ on the $n+1$-qubit register, we get $\ket{\tilde{F}^\top b}$. More specifically, these two steps lead to

\[
\begin{aligned}
\ket{0}_{n+1} \ket{0}_m \rightarrow & \left( \frac{\|{F^{(0)}}^\top Z^\top b\|}{\|{\tilde{F}}^\top b\|} \ket{0}_{n+1} + \sum_{i=1}^N \frac{\|{F^{(i)}}^\top R_L^{(i)} b\|}{\|{\tilde{F}}^\top b\|} \ket{i}_{{n+1}} \right) \ket{0}_{m}, \\
\rightarrow & \frac{\|{F^{(0)}}^\top Z^\top b\|}{\|{\tilde{F}}^\top b\|} \ket{0} \ket{{F^{(0)}}^\top Z^\top b } + \sum_{i=1}^N \frac{\|{F^{(i)}}^\top R_L^{(i)} b\|}{\|{\tilde{F}}^\top b\|} \ket{i} \ket{{F^{(i)}}^\top R_L^{(i)} b } = \ket{\tilde{F}^\top b}.
\end{aligned}
\]

\subsection{Domain decomposition with local BPX solvers}
\label{sec:quantumDDBPX}
One possible choice of local solver is the Bramble--Pasciak--Xu (BPX) preconditioner \cite{zbMATH04197323} which is spectrally equivalent to $A^{(i)}$, and for which \cite{Deiml_2025} proposes a quantum implementation. We first define the BPX preconditioner in subdomain $\Omega^{(i)}$.

\begin{definition}[BPX preconditioner {\cite[Section 5]{Deiml_2025}}] Let $i \in \{ 1, \dots, N \}$ be the index of a subdomain. Let $\mathcal{T}^{(i)}_1 \in \cdots \in \mathcal{T}^{(i)}_L$ be a sequence of nested grids of the unit hypercube $\Omega^{(i)}$: each grid $\mathcal{T}^{(i)}_l$ is of mesh size $2^{-l}$, and is equipped with the $\mathbb Q_1$ finite element space $V_l^{(i)}$ of $d$-linear functions that are zero on the boundary and globally continuous. (For index $L$, the notation coincides with that of Section~\ref{sec:geom}.)

If, for every level $l \in \{ 1, \dots, L \}$, $\left\{ \phi_{j,l}^{(i)} \right\}_{1 \leq j \leq \text{dim } V_l^{(i)}}$ is a basis of $V_l^{(i)}$, and $\mathcal I_{l,L}$ is the matrix representation of the embedding operator $V_l^{(i)} \hookrightarrow V_L^{(i)}$, then, the BPX preconditioner is  
\begin{equation*}
 H^{(i)} := F^{(i)} {F^{(i)}}^\top; \, F^{(i)} := (2^{-(2-d)/2} \mathcal I_{1,L} | \cdots |2^{-l(2-d)/2} \mathcal I_{l,L} | \cdots |2^{-L(2-d)/2} \mathcal I_{L,L}  ).
\end{equation*}
\label{def:BPX_precond}
\end{definition}
Remark that $F^{(i)} \in\mathbb{R}^{\text{dim } V_L^{(i)} \times \sum_{l=1}^L \text{dim } V_l^{(i)} }$. We next state the spectral equivalence result.
\begin{lemma}
For each subdomain $1\leq i \leq N$, let $H^{(i)} = F^{(i)} {F^{(i)}}^\top$ be the split BPX matrix from \cref{def:BPX_precond}. The eigenvalues of $H^{(i)} A^{(i)}$ are in an interval $[\mu_{\min}, \mu_{\max}]$ with bounds independent of the mesh size (parametrized by $L$). 
\label{lem:eigenvalues_bpx}
\end{lemma}
\begin{proof}The matrix $A^{(i)}$ which we defined by extracting rows and columns from the global matrix $A$ is also the matrix of the Poisson problem on $\Omega^{(i)}$ with homogeneous Dirichlet boundary conditions. For this reason, \cref{lem:eigenvalues_bpx} is the usual optimality result for BPX from the literature. It is stated in \cite[Lemma 5.1]{Deiml_2025} that
\[
 \kappa({F^{(i)}}^\top  A^{(i)} F^{(i)}) = \mathcal O (1) 
\]
with no hidden $L$ in the $ \mathcal O$. Although the stated result is for the condition number the cited proofs \textit{e.g.} \cite{zbMATH00168423} prove independently that $\mu_{\min}$ and $\mu_{\max}$ are $\mathcal O (1) $. 
\end{proof}

Another advantage of choosing BPX as our local solver is that, the block-encodings $U_{C_L^{(i)} F^{(i)}}$ in the Assumptions of~\ref{thm:th_BE_DD} are exactly the block-encodings proposed in \cite{Deiml_2025} so we can write the following result. 

\begin{theorem}[Block-encoding of two-level Additive Schwarz with BPX local solvers]
Let $U_{D_\rho}$ be a block-encoding of $D_\rho$ and let $U_{C_L Z F^{(0)}}$ be a block-encoding of $C_L Z F^{(0)}$. Using BPX for the local solves, there exists a block-encoding of $\tilde{F}^\top A \tilde{F}$ with normalization factor
\begin{equation}
\alpha (U_{\tilde{F}^\top A \tilde{F}}) = \alpha(U_{D_\rho}) \left(4 N d L + \alpha \left(U_{C_L Z F^{(0)}} \right)^2 \right),
\end{equation}
and subnormalization factor
\begin{equation}
\tilde{\alpha} (U_{\tilde{F}^\top A \tilde{F}}) \leq \kappa(D_\rho) \tilde{\alpha}(U_{D_\rho}) \left( N d (L+C_d) + \tilde{\alpha} \left(U_{C_L Z F^{(0)}} \right)^2 \right),
\end{equation}
where $C_d$ is a constant that depends only on $d$.
\label{thm:th_BE_DD_BPX}
\end{theorem}

\begin{proof}
The normalization and subnormalization factors of the block-encodings of $C_L^{(i)} F^{(i)}$ are bounded by \cite[Theorem 6.3]{Deiml_2025} as follows 
\begin{subequations}
\begin{align}
    \alpha\left(U_{C_L^{(i)} F^{(i)}}\right)^2 & = 4dL, \\
    \tilde{\alpha}\left(U_{C_L^{(i)} F^{(i)}}\right)^2 & \leq d(L+C_d),
\end{align}
\end{subequations}
where $C_d$ is a constant that depends only on $d$. We conclude by injecting these bounds into \cref{thm:th_BE_DD}. 
\end{proof}

\begin{remark}[Block-encoding of $C_L Z F^{(0)}$]
We have not yet proposed an optimized block-encoding of $C_L Z F^{(0)}$. We recall that the columns of $Z$ span the coarse space, $C_L$ is the discontinuous gradient operator and $F^{(0)}$ factorizes the coarse space ($ (Z^\top A Z)^{-1} = F^{(0)} {F^{(0)}}^\top $). We do not anticipate this to be problematic. 
\end{remark}

We are interested in solving the linear system \cref{eq:prec_system} and computing some quantities of interest.
\begin{theorem}[Complexity of domain decomposition with local BPX solves]
Let $U_{D_\rho}$ and $U_{C_L Z F^{(0)}}$ be block-encodings of $D_\rho$ and $C_L Z F^{(0)}$. Let $w \in \mathbb{C}^{\mathcal{N}}$. Let $U_{\tilde{b}}$ and $U_{\tilde{w}}$ be oracles preparing $\ket{\tilde{b}} := \ket{{\tilde F}^\top {b}} $ and $\ket{\tilde{w}}:= \ket{{\tilde F}^\top {w}}$ respectively. 

Let $y$ be the solution to the preconditioned linear system \eqref{eq:prec_system}. We recall from Section~\ref{sec:geom} that $L$ is the mesh parameter, $N$ is the number of subdomains in the partition of $\Omega$, and $\delta$ is the amount of overlap between subdomains.

Computing the quantity $\braket{\tilde{w}}{y}$ with precision $\epsilon$ requires 
\[
\bigO{ \left( N L + \tilde{\alpha} \left(U_{C_L Z F^{(0)}} \right)^2 \right) \frac{1}{\epsilon \delta^2} \log\left(\frac{1}{\delta \epsilon}\right)}
\]
queries to $U_{\tilde{F}^\top A \tilde{F}}$ and its inverse and $\bigO{\frac{1}{\epsilon \delta}}$ queries to $U_{\tilde{b}}$ and $U_{\tilde{w}}$. The notation $\mathcal O$ does not hide any constants depending on $N$, $L$, $\delta$ or $\epsilon$. 
\label{thm:QLSS_DD_BPX}
\end{theorem}

\begin{proof}
\cref{th:th_kappa_DD_BPX} and \cref{lem:eigenvalues_bpx} give $\kappa(\tilde{F}^\top A \tilde{F}) = \bigO{\frac{\lambda_{\max}(\Htwo A)}{\lambda_{\min}(\Htwo A)}}$, \textit{i.e.}, replacing the exact local solves by their BPX approximation does not introduce any extra dependencies in $N$, $L$ or $\delta$.
Then, the spectral bounds for $\Htwo A$ from \cref{th:DDtheory} give $\kappa(\tilde{F}^\top A \tilde{F}) = \bigO{\frac{1}{\delta}}$.
By injecting this, and the bound on the subnormalization constant from \cref{thm:th_BE_DD_BPX}, into \cref{lem:qlss_qsvt}, we obtain that the QSVT-based QLSS  outputs a state $\epsilon$-close to the normalized solution of the preconditioned system $\ket{y} := \frac{ y}{\| y \|}$ using $\bigO{ \left( N L + \tilde{\alpha} \left(U_{C_L Z F^{(0)}} \right)^2 \right) \frac{1}{\delta^2} \log\left(\frac{1}{\delta \epsilon}\right)}$ queries to $U_A$ and $U_A^\dag$ and $\bigO{\frac{1}{\delta}}$ queries to $U_{\tilde{b}}$ and $U_{\tilde{w}}$. The quantity $\kappa(D_\rho) \tilde{\alpha}(U_{D_\rho})$ related to the diffusion parameter is in the $\mathcal O$ following \cite[Proposition 3.4]{Deiml_2025}. 
Using amplitude estimation and a swap test, the computation of $\braket{\tilde{w}}{y}$ up to precision $\epsilon$ multiplies the previous result by a factor $\epsilon^{-1}$. 
\end{proof}

We remark that the complexity scales linearly with the number of subdomains. This is due to the particular construction of the block-encoding of the preconditioned system. In domain decomposition, it is standard to eliminate $N$ and instead have a dependency in the number of colors $N_C$ needed to color
the subdomains so that two overlapping subdomains never share the same color (see \textit{e.g}, the spectral result in Theorem~\ref{th:DDtheory}). We leave this for future work in the quantum formalism.

\section{Numerical simulations}
\label{sec:numerical}
We now turn to numerical simulations for two-level Additive Schwarz with BPX local solves. We have not included a coarse space in these simulations. The quantum circuits are emulated in Qiskit \cite{qiskit2024} for a one-dimensional domain ($d=1$). The code associated with \cite{Deiml_2025} is used for BPX local solvers. The inner product is split as $\left[(\tilde{F}^\top C_L^\top)^+ \tilde{w} \right]^\top (\tilde{F}^\top C_L^\top)^+ \tilde{b}$ \cite{Deiml_2025}. The linear systems are inverted using QSVT with the polynomial from \cite[Lemma 17-19]{zbMATH06823707} and \cite[Lemma 40]{Gily_n_2019}. This polynomial is normalized to satisfy the requirements of QSVT. Its degree, denoted by $D$, is computed using the quantity $\kappa(\tilde{F}^\top C_L^\top) \tilde{\alpha}(\tilde{F}^\top C_L^\top)$ and the target QSVT accuracy $\epsilon=10^{-6}$. The angles for the polynomials are computed using the library pysqp \cite{mrtc_unification_21,cdghs_finding_qsp_angles_20,dmwl_efficient_phases_21}. The domain $\Omega$ depends on the number of subdomains. We choose the right-hand side as $f=1$ so that $b=M (1 \cdots 1)^\top$ where $M$ is the mass matrix. We choose $w=b$ for the computation of the inner product. In that case, only the norm of the quantum solution is estimated from the probability of success of QSVT. The norm is multiplied by the normalization constant of the polynomial and by the norm of the right-hand side and divided by the normalization constant. The mesh size is $2^{-L}$ with $L=4$. Two values of numbers of subdomains are considered $N_1 = 2$ and $N_1=4$. The simulations are run $100$ times, each with $10^6$ samples. The results are reported in \cref{tab:tab_inner_product}. We check that the computed mean value $\mathcal{Q}_{\text{mean}}$ is a good approximation of $\mathcal{Q}_{\text{ref}}=w^\top b$ computed classically. We also report the effective condition number $\kappa(\tilde{F}^\top C_L^\top) \tilde{\alpha}(U_{\tilde{F}^\top C_L^\top})$ and the degree of the inversion polynomial. Both depend on  the number of subdomains $N_1$. Adding the coarse correction would eliminate this dependency.

\begin{table}[htbp]
\caption{Results from a 1D simulation. The classical value of the inner product is $\mathcal{Q}_{\text{ref}}=w^\top b$. $\mathcal{Q}_{\text{mean}}$ denotes the mean value of the $100$ runs. $\mathcal{Q}_{2.5}$ and $\mathcal{Q}_{97.5}$ are the $2.5$ and $97.5$ percentiles respectively.}\label{tab:tab_inner_product}
\begin{center}
\begin{tabular}{|c | c | c | c | c | c | c |}
 \hline
 $N_1$ & $\kappa(\tilde{F}^\top C_L^\top) \tilde{\alpha}(U_{\tilde{F}^\top C_L^\top})$ & $D$ & $\mathcal{Q}_{\text{ref}}$ & $\mathcal{Q}_{\text{mean}}$ & $\mathcal{Q}_{0.25}$ & $\mathcal{Q}_{0.75}$ \\ [0.5ex]
 \hline\hline
 $2$ & $10.47$ & $201$ & $0.549 $ & $0.548$ & $0.540$ & $0.556$  \\
 \hline
 $4$ & $26.75$ & $549$ & $3.97$ & $3.97$ & $3.90$ & $4.02$  \\
 \hline
\end{tabular}
\end{center}
\end{table}

\section{Conclusions}
\label{sec:conclusions}
 In this work, we have introduced the first quantum domain decomposition preconditioner. We have demonstrated that it is feasible to block-encode the preconditioned problem that results from preconditioning Poisson by the Additive Schwarz preconditioner. We have studied the overall complexity of measuring a quantity of interest. In future work we plan to investigate other choices for the local solves: sparse approximate inverses, circulant preconditioners, Neumann solves\dots Particular attention should also be paid to the quantum coarse correction. Finally, we are optimistic that the linear dependency of the complexity in the number of subdomains can be replaced by the coloring constant.

\bibliographystyle{siamplain}
\bibliography{references}

@article{mrtc_unification_21,
  title = {Grand Unification of Quantum Algorithms},
  author = {Martyn, John M. and Rossi, Zane M. and Tan, Andrew K. and Chuang, Isaac L.},
  journal = {PRX Quantum},
  volume = {2},
  issue = {4},
  pages = {040203},
  numpages = {40},
  year = {2021},
  month = {Dec},
  publisher = {American Physical Society},
  doi = {10.1103/PRXQuantum.2.040203},
  url = {https://link.aps.org/doi/10.1103/PRXQuantum.2.040203}
}

@article{dmwl_efficient_phases_21,
  title={Efficient phase-factor evaluation in quantum signal processing},
  volume={103},
  ISSN={2469-9934},
  url={http://dx.doi.org/10.1103/PhysRevA.103.042419},
  DOI={10.1103/physreva.103.042419},
  number={4},
  journal={Physical Review A},
  publisher={American Physical Society (APS)},
  author={Dong, Yulong and Meng, Xiang and Whaley, K. Birgitta and Lin, Lin},
  year={2021},
  month=apr
}

@article{cdghs_finding_qsp_angles_20,
  title={Finding Angles for Quantum Signal Processing with Machine Precision},
  author={Rui Chao and Dawei Ding and Andras Gilyen and Cupjin Huang and Mario Szegedy},
  year={2020},
  eprint={2003.02831},
  archivePrefix={arXiv},
  primaryClass={quant-ph}
}

@article{zbMATH01592007,
 author = {Klawonn, Axel and Widlund, Olof B.},
 title = {{FETI} and {Neumann}-{Neumann} iterative substructuring methods: {Connections} and new results},
 fjournal = {Communications on Pure and Applied Mathematics},
 journal = {Commun. Pure Appl. Math.},
 issn = {0010-3640},
 volume = {54},
 number = {1},
 pages = {57--90},
 year = {2001},
 language = {English},
 doi = {10.1002/1097-0312(200101)54:1<57::AID-CPA3>3.0.CO;2-D},
 url = {publica.fraunhofer.de/documents/B-73220.html},
 zbMATH = {1592007},
 Zbl = {1023.65120}
}

@book{zbMATH01953444,
 author = {Saad, Yousef},
 title = {Iterative methods for sparse linear systems.},
 edition = {2nd ed.},
 isbn = {0-89871-534-2},
 year = {2003},
 publisher = {Philadelphia, PA: SIAM Society for Industrial {and} Applied Mathematics},
 language = {English},
 zbMATH = {1953444},
 Zbl = {1031.65046}
}

@book{zbMATH00635667,
 author = {Quarteroni, Alfio and Valli, Alberto},
 title = {Numerical approximation of partial differential equations},
 fseries = {Springer Series in Computational Mathematics},
 series = {Springer Ser. Comput. Math.},
 issn = {0179-3632},
 volume = {23},
 isbn = {3-540-57111-6},
 year = {1994},
 publisher = {Berlin: Springer},
 language = {English},
 zbMATH = {635667},
 Zbl = {0803.65088}
}

@article{Deiml_2025,
   title={Quantum realization of the finite element method},
   ISSN={0025-5718},
   url={http://dx.doi.org/10.1090/mcom/4137},
   DOI={10.1090/mcom/4137},
   journal={Mathematics of Computation},
   publisher={American Mathematical Society (AMS)},
   author={Deiml, M. and Peterseim, D.},
   year={2025},
   month=aug }

@article{Harrow_2009,
   title={Quantum Algorithm for Linear Systems of Equations},
   volume={103},
   ISSN={1079-7114},
   url={http://dx.doi.org/10.1103/PhysRevLett.103.150502},
   DOI={10.1103/physrevlett.103.150502},
   number={15},
   journal={Physical Review Letters},
   publisher={American Physical Society (APS)},
   author={Harrow, Aram W. and Hassidim, Avinatan and Lloyd, Seth},
   year={2009},
   month=oct }

@incollection{zbMATH06070940,
 author = {Ambainis, Andris},
 title = {Variable time amplitude amplification and quantum algorithms for linear algebra problems},
 booktitle = {STACS 2012. 29th international symposium on theoretical aspects of computer science, Paris, France, February 29th -- March 3rd, 2012},
 isbn = {978-3-939897-35-4},
 pages = {636--647},
 year = {2012},
 publisher = {Wadern: Schloss Dagstuhl -- Leibniz Zentrum f{\"u}r Informatik},
 language = {English},
 doi = {10.4230/LIPIcs.STACS.2012.636},
 zbMATH = {6070940},
 Zbl = {1245.68084}
}

@article{zbMATH06823707,
 author = {Childs, Andrew M. and Kothari, Robin and Somma, Rolando D.},
 title = {Quantum algorithm for systems of linear equations with exponentially improved dependence on precision},
 fjournal = {SIAM Journal on Computing},
 journal = {SIAM J. Comput.},
 issn = {0097-5397},
 volume = {46},
 number = {6},
 pages = {1920--1950},
 year = {2017},
 language = {English},
 doi = {10.1137/16M1087072},
 zbMATH = {6823707},
 Zbl = {1383.68034}
}

@inproceedings{Gily_n_2019, 
series={STOC ’19},
   title={Quantum singular value transformation and beyond: exponential improvements for quantum matrix arithmetics},
   url={http://dx.doi.org/10.1145/3313276.3316366},
   DOI={10.1145/3313276.3316366},
   booktitle={Proceedings of the 51st Annual ACM SIGACT Symposium on Theory of Computing},
   publisher={ACM},
   author={Gilyén, András and Su, Yuan and Low, Guang Hao and Wiebe, Nathan},
   year={2019},
  pages={193–204},
   collection={STOC ’19} }

@article{PhysRevLett.122.060504,
  title = {Quantum Algorithms for Systems of Linear Equations Inspired by Adiabatic Quantum Computing},
  author = {Suba\ifmmode \mbox{\c{s}}\else \c{s}\fi{}\ifmmode \imath \else \i \fi{}, Yi\ifmmode \breve{g}\else \u{g}\fi{}it and Somma, Rolando D. and Orsucci, Davide},
  journal = {Phys. Rev. Lett.},
  volume = {122},
  issue = {6},
  pages = {060504},
  numpages = {5},
  year = {2019},
  month = {Feb},
  publisher = {American Physical Society},
  doi = {10.1103/PhysRevLett.122.060504},
  url = {https://link.aps.org/doi/10.1103/PhysRevLett.122.060504}
}

@article{Lin2020optimalpolynomial,
  doi = {10.22331/q-2020-11-11-361},
  url = {https://doi.org/10.22331/q-2020-11-11-361},
  title = {Optimal polynomial based quantum eigenstate filtering with application to solving quantum linear systems},
  author = {Lin, Lin and Tong, Yu},
  journal = {{Quantum}},
  issn = {2521-327X},
  publisher = {{Verein zur F{\"{o}}rderung des Open Access Publizierens in den Quantenwissenschaften}},
  volume = {4},
  pages = {361},
  month = nov,
  year = {2020}
}

@article{10.1145/3498331,
author = {An, Dong and Lin, Lin},
title = {Quantum Linear System Solver Based on Time-optimal Adiabatic Quantum Computing and Quantum Approximate Optimization Algorithm},
year = {2022},
issue_date = {June 2022},
publisher = {Association for Computing Machinery},
address = {New York, NY, USA},
volume = {3},
number = {2},
url = {https://doi.org/10.1145/3498331},
doi = {10.1145/3498331},
journal = {ACM Transactions on Quantum Computing},
month = mar,
articleno = {5},
numpages = {28}
}

@article{PRXQuantum.3.040303,
  title = {Optimal Scaling Quantum Linear-Systems Solver via Discrete Adiabatic Theorem},
  author = {Costa, Pedro C.S. and An, Dong and Sanders, Yuval R. and Su, Yuan and Babbush, Ryan and Berry, Dominic W.},
  journal = {PRX Quantum},
  volume = {3},
  issue = {4},
  pages = {040303},
  numpages = {54},
  year = {2022},
  month = {Oct},
  publisher = {American Physical Society},
  doi = {10.1103/PRXQuantum.3.040303},
  url = {https://link.aps.org/doi/10.1103/PRXQuantum.3.040303}
}

@article{Low2026quantumlinearsystem,
  doi = {10.22331/q-2026-03-23-2041},
  url = {https://doi.org/10.22331/q-2026-03-23-2041},
  title = {Quantum linear system algorithm with optimal queries to initial state preparation},
  author = {Low, Guang Hao and Su, Yuan},
  journal = {{Quantum}},
  issn = {2521-327X},
  publisher = {{Verein zur F{\"{o}}rderung des Open Access Publizierens in den Quantenwissenschaften}},
  volume = {10},
  pages = {2041},
  month = mar,
  year = {2026}
}

@article{BravoPrieto2023variationalquantum,
  doi = {10.22331/q-2023-11-22-1188},
  url = {https://doi.org/10.22331/q-2023-11-22-1188},
  title = {Variational {Q}uantum {L}inear {S}olver},
  author = {Bravo-Prieto, Carlos and LaRose, Ryan and Cerezo, M. and Subasi, Yigit and Cincio, Lukasz and Coles, Patrick J.},
  journal = {{Quantum}},
  issn = {2521-327X},
  publisher = {{Verein zur F{\"{o}}rderung des Open Access Publizierens in den Quantenwissenschaften}},
  volume = {7},
  pages = {1188},
  month = nov,
  year = {2023}
}

@misc{morales2025quantumlinearsolverssurvey,
      title={Quantum Linear System Solvers: A Survey of Algorithms and Applications}, 
      author={Mauro E. S. Morales and Lirandë Pira and Philipp Schleich and Kelvin Koor and Pedro C. S. Costa and Dong An and Alán Aspuru-Guzik and Lin Lin and Patrick Rebentrost and Dominic W. Berry},
      year={2025},
      eprint={2411.02522},
      archivePrefix={arXiv},
      primaryClass={quant-ph},
      url={https://arxiv.org/abs/2411.02522}, 
}

@article{PhysRevLett.110.250504,
  title = {Preconditioned Quantum Linear System Algorithm},
  author = {Clader, B. D. and Jacobs, B. C. and Sprouse, C. R.},
  journal = {Phys. Rev. Lett.},
  volume = {110},
  issue = {25},
  pages = {250504},
  numpages = {5},
  year = {2013},
  month = {Jun},
  publisher = {American Physical Society},
  doi = {10.1103/PhysRevLett.110.250504},
  url = {https://link.aps.org/doi/10.1103/PhysRevLett.110.250504}
}

@article{PhysRevA.98.062321,
  title = {Quantum circulant preconditioner for a linear system of equations},
  author = {Shao, Changpeng and Xiang, Hua},
  journal = {Phys. Rev. A},
  volume = {98},
  issue = {6},
  pages = {062321},
  numpages = {9},
  year = {2018},
  month = {Dec},
  publisher = {American Physical Society},
  doi = {10.1103/PhysRevA.98.062321},
  url = {https://link.aps.org/doi/10.1103/PhysRevA.98.062321}
}

@article{PhysRevA.93.032324,
  title = {Quantum algorithms and the finite element method},
  author = {Montanaro, Ashley and Pallister, Sam},
  journal = {Phys. Rev. A},
  volume = {93},
  issue = {3},
  pages = {032324},
  numpages = {14},
  year = {2016},
  month = {Mar},
  publisher = {American Physical Society},
  doi = {10.1103/PhysRevA.93.032324},
  url = {https://link.aps.org/doi/10.1103/PhysRevA.93.032324}
}

@article{PhysRevA.104.032422,
  title = {Fast inversion, preconditioned quantum linear system solvers, fast Green's-function computation, and fast evaluation of matrix functions},
  author = {Tong, Yu and An, Dong and Wiebe, Nathan and Lin, Lin},
  journal = {Phys. Rev. A},
  volume = {104},
  issue = {3},
  pages = {032422},
  numpages = {33},
  year = {2021},
  month = {Sep},
  publisher = {American Physical Society},
  doi = {10.1103/PhysRevA.104.032422},
  url = {https://link.aps.org/doi/10.1103/PhysRevA.104.032422}
}

@misc{childs2026quantumalgorithmsheterogeneouspdes,
      title={Quantum Algorithms for Heterogeneous PDEs: The Neutron Diffusion Eigenvalue Problem}, 
      author={Andrew M. Childs and Lincoln Johnston and Brian Kiedrowski and Mahathi Vempati and Jeffery Yu},
      year={2026},
      eprint={2604.05098},
      archivePrefix={arXiv},
      primaryClass={quant-ph},
      url={https://arxiv.org/abs/2604.05098}, 
}

@misc{jin2025quantumpreconditioningmethodlinear,
      title={Quantum preconditioning method for linear systems problems via Schr\"odingerization}, 
      author={Shi Jin and Nana Liu and Chuwen Ma and Yue Yu},
      year={2025},
      eprint={2505.06866},
      archivePrefix={arXiv},
      primaryClass={math.NA},
      url={https://arxiv.org/abs/2505.06866}, 
}

@misc{draper2000additionquantumcomputer,
      title={Addition on a Quantum Computer}, 
      author={Thomas G. Draper},
      year={2000},
      eprint={quant-ph/0008033},
      archivePrefix={arXiv},
      primaryClass={quant-ph},
      url={https://arxiv.org/abs/quant-ph/0008033}, 
}

@misc{busaleh2026mitigatingbarrenplateausdomain,
      title={Mitigating Barren Plateaus via Domain Decomposition in Variational Quantum Algorithms for Nonlinear PDEs}, 
      author={Laila S. Busaleh and Jeonghyeuk Kwon and Orlane Zang and Muhammad Hassan and Yvon Maday},
      year={2026},
      eprint={2603.24523},
      archivePrefix={arXiv},
      primaryClass={math.NA},
      url={https://arxiv.org/abs/2603.24523}, 
}

@article{Yuan_2023,
doi = {10.1088/1367-2630/acfd52},
url = {https://doi.org/10.1088/1367-2630/acfd52},
year = {2023},
month = {oct},
publisher = {IOP Publishing},
volume = {25},
number = {10},
pages = {103011},
author = {Yuan, Yewei and Wang, Chao and Wang, Bei and Chen, Zhao-Yun and Dou, Meng-Han and Wu, Yu-Chun and Guo, Guo-Ping},
title = {An improved QFT-based quantum comparator and extended modular arithmetic using one ancilla qubit},
journal = {New Journal of Physics}
}

@article{10.5555/2011517.2011525,
    author = {Beauregard, Stephane},
    title = {Circuit for Shor's algorithm using 2n+3 qubits},
    year = {2003},
    issue_date = {March 2003},
    publisher = {Rinton Press, Incorporated},
    address = {Paramus, NJ},
    volume = {3},
    number = {2},
    issn = {1533-7146},
    journal = {Quantum Info. Comput.},
    month = mar,
    pages = {175–185},
    numpages = {11}
    }

@misc{balazi2026quantumenhancednumericalhomogenization,
      title={Quantum Enhanced Numerical Homogenization}, 
      author={Loïc Balazi and Matthias Deiml and Daniel Peterseim},
      year={2026},
      eprint={2603.28521},
      archivePrefix={arXiv},
      primaryClass={math.NA},
      url={https://arxiv.org/abs/2603.28521}, 
}

@inproceedings{Koska:2025ucz,
    author = "Koska, Oc{\'e}ane and Baboulin, Marc and Gazda, Arnaud",
    title = "{A mixed-precision quantum-classical algorithm for solving linear systems}",
    booktitle = "{2025 IEEE International Parallel and Distributed Processing Symposium Workshops}",
    eprint = "2502.02212",
    archivePrefix = "arXiv",
    primaryClass = "quant-ph",
    doi = "10.1109/IPDPSW66978.2025.00081",
    month = "2",
    year = "2025"
}

@book{zbMATH02113718,
 author = {Toselli, Andrea and Widlund, Olof},
 title = {Domain decomposition methods -- algorithms and theory.},
 fseries = {Springer Series in Computational Mathematics},
 series = {Springer Ser. Comput. Math.},
 issn = {0179-3632},
 volume = {34},
 isbn = {3-540-20696-5},
 year = {2005},
 publisher = {Berlin: Springer},
 language = {English},
 zbMATH = {2113718},
 Zbl = {1069.65138}
}

@misc{zbMATH04197323,
 author = {Bramble, J. H. and Pasciak, J. E. and Xu, J.},
 title = {Parallel multilevel preconditioners},
 year = {1990},
 language = {English},
 howpublished = {Numerical analysis, {Proc}. 13th {Biennial} {Conf}., {Dundee}/{UK} 1989, {Pitman} {Res}. {Notes} {Math}. {Ser}. 228, 23-39 (1990).},
 zbMATH = {4197323},
 Zbl = {0725.65095}
}

@misc{lin2026lecturenotes,
      title={Quantum Algorithms for Scientific Computation}, 
      author={Lin, Lin and Wiebe, Nathan},
      year={2026},
      url={https://math.berkeley.edu/~linlin/qasc/live_notes_0429.pdf}, 
}

@misc{bagherimehrab2025fastquantumalgorithmdifferential,
      title={Fast quantum algorithm for differential equations}, 
      author={Mohsen Bagherimehrab and Kouhei Nakaji and Nathan Wiebe and Gavin K. Brennen and Barry C. Sanders and Alán Aspuru-Guzik},
      year={2025},
      eprint={2306.11802},
      archivePrefix={arXiv},
      primaryClass={quant-ph},
      url={https://arxiv.org/abs/2306.11802}, 
}

@Article{zbMATH06303461,
 Author = {Spillane, N. and Dolean, V. and Hauret, P. and Nataf, F. and Pechstein, C. and Scheichl, R.},
 Title = {Abstract robust coarse spaces for systems of {PDEs} via generalized eigenproblems in the overlaps},
 FJournal = {Numerische Mathematik},
 Journal = {Numer. Math.},
 ISSN = {0029-599X},
 Volume = {126},
 Number = {4},
 Pages = {741--770},
 Year = {2014},
 Language = {English},
 DOI = {10.1007/s00211-013-0576-y},
 URL = {strathprints.strath.ac.uk/44827/1/10.1007_s00211_013_0576_y.pdf},
 zbMATH = {6303461},
 Zbl = {1291.65109},
 Note={\url{https://people.bath.ac.uk/masrs/2011-07.pdf}},
}

@article{spillane_geneo_2025,
        title = {{GenEO} spectral coarse spaces in {SPD} domain decomposition},
        issn = {1572-9265},
        url = {https://doi.org/10.1007/s11075-025-02166-x},
        doi = {10.1007/s11075-025-02166-x},
        journal = {Numerical Algorithms},
        author = {Spillane, Nicole},
        NOTE = {\url{https://hal.science/hal-03186276/document}},
        year = {2025},
}

@misc{gu2025quantumsimulationhelmholtzequations,
      title={Quantum simulation of Helmholtz equations via Schr{\"o}dingerization}, 
      author={Anjiao Gu and Shi Jin and Chuwen Ma},
      year={2025},
      eprint={2507.23547},
      archivePrefix={arXiv},
      primaryClass={math.NA},
      url={https://arxiv.org/abs/2507.23547}, 
}

@misc{arXiv:2512.06166,
 author = {Boou Jiang and Jongho Park and Jinchao Xu},
 title = {A polynomial dimension-dependence analysis of {Bramble}--{Pasciak}--{Xu} preconditioners},
 year = {2025},
 howpublished = {Preprint, {arXiv}:2512.06166 [math.{NA}] (2025)},
 url = {https://arxiv.org/abs/2512.06166},
 arXiv = {arXiv:2512.06166}
}

@TechReport{petsc-user-ref,
       author = {Satish Balay and Shrirang Abhyankar and Mark~F. Adams and Jed Brown and Peter Brune
                 and Kris Buschelman and Lisandro Dalcin and Alp Dener and Victor Eijkhout and William~D. Gropp
                 and Dmitry Karpeyev and Dinesh Kaushik and Matthew~G. Knepley and Dave~A. May and Lois Curfman McInnes
                 and Richard Tran Mills and Todd Munson and Karl Rupp and Patrick Sanan
                 and Barry~F. Smith and Stefano Zampini and Hong Zhang and Hong Zhang},
       title  = {{PETS}c Users Manual},
       institution = {Argonne National Laboratory},
       year   = 2021,
       number = {ANL-95/11 - Revision 3.15},
       url    = {https://www.mcs.anl.gov/petsc}
     }

@article{zbMATH05172548,
 author = {Graham, I. G. and Lechner, P. O. and Scheichl, R.},
 title = {Domain decomposition for multiscale {PDEs}},
 fjournal = {Numerische Mathematik},
 journal = {Numer. Math.},
 issn = {0029-599X},
 volume = {106},
 number = {4},
 pages = {589--626},
 year = {2007},
 language = {English},
 doi = {10.1007/s00211-007-0074-1},
 zbMATH = {5172548},
 Zbl = {1141.65084}
}

@article{zbMATH05029831,
 author = {Gander, Martin J.},
 title = {Optimized {Schwarz} methods},
 fjournal = {SIAM Journal on Numerical Analysis},
 journal = {SIAM J. Numer. Anal.},
 issn = {0036-1429},
 volume = {44},
 number = {2},
 pages = {699--731},
 year = {2006},
 language = {English},
 doi = {10.1137/S0036142903425409},
 url = {archive-ouverte.unige.ch/unige:171407},
 zbMATH = {5029831},
 Zbl = {1117.65165}
}

@article{zbMATH05651315,
 author = {Pechstein, Clemens and Scheichl, Robert},
 title = {Scaling up through domain decomposition},
 fjournal = {Applicable Analysis},
 journal = {Appl. Anal.},
 issn = {0003-6811},
 volume = {88},
 number = {10-11},
 pages = {1589--1608},
 year = {2009},
 language = {English},
 doi = {10.1080/00036810903157204},
 zbMATH = {5651315},
 Zbl = {1281.76040}
}

@article{zbMATH00168423,
 author = {Xu, Jinchao},
 title = {Iterative methods by space decomposition and subspace correction},
 fjournal = {SIAM Review},
 journal = {SIAM Rev.},
 issn = {0036-1445},
 volume = {34},
 number = {4},
 pages = {581--613},
 year = {1992},
 language = {English},
 doi = {10.1137/1034116},
 zbMATH = {168423},
 Zbl = {0788.65037}
}

@misc{qiskit2024,
      title={Quantum computing with {Q}iskit},
      author={Javadi-Abhari, Ali and Treinish, Matthew and Krsulich, Kevin and Wood, Christopher J. and Lishman, Jake and Gacon, Julien and Martiel, Simon and Nation, Paul D. and Bishop, Lev S. and Cross, Andrew W. and Johnson, Blake R. and Gambetta, Jay M.},
      year={2024},
      doi={10.48550/arXiv.2405.08810},
      eprint={2405.08810},
      archivePrefix={arXiv},
      primaryClass={quant-ph}
}
\end{document}